\documentclass[11pt]{amsart}
\usepackage{amsmath,amssymb,amscd,verbatim,color}
\let\Cal=\mathcal

\newtheorem{Theorem}{Theorem}[section]
\newtheorem{Lemma}{Lemma}[section]
\newtheorem{Proposition}{Proposition}[section]
\newtheorem{Corollary}{Corollary}[section]
\newtheorem{Remark}{Remark}[section]

\newtheorem{Definition}{Definition}[section]

\numberwithin{equation}{section}

\allowdisplaybreaks

\newcommand{\R}{{\rm P^T}}

\def\R{{\Bbb R}}
\def\B{{\Bbb B}}
\def\T{{\Bbb T}}

\def\S{{\Bbb S}}
\def\E{{\Bbb E}}

\def \rmd {{\rm d}}
\def \rme {{\rm e}}

\def\bc{\begin{center}}
\def\ec{\end{center}}

\def\di{\displaystyle\int}
\def\a1{(a_1, a_2, \cdots, a_n)}

\def\cal{\Cal}

\allowdisplaybreaks

\begin{document}

\title[ Concentration and limit behaviors of
stationary measures]{ Concentration and limit behaviors of
stationary measures}

\author[W. Huang]{Wen Huang}
\address{W. Huang: Department of Mathematics, Sichuan University,
Chengdu, Sichuan 610064, PRC, and School of Mathematical Sciences, University of Science and Technology of
China, Hefei, Anhui 230026, PRC}
\email{wenh@mail.ustc.edu.cn}
\author[M. Ji] {Min Ji}
\address{M. Ji: Academy of Mathematics and System Sciences,
Chinese Academy of Sciences, Beijing 100080, PRC}
\email{jimin@math.ac.cn}
\author[Z. Liu]{Zhenxin Liu}
\address{Z. Liu: School of Mathematical Sciences, Dalian University of Technology,
Dalian 116024, PRC} \email{zxliu@dlut.edu.cn}
\author[Y. Yi]{Yingfei Yi}
\address{Y. Yi: Department of Mathematical
\& Statistical Sci, University of Alberta, Edmonton£¬ Alberta,
Canada T6G 2G1,  School of Mathematics, Jilin University, Changchun
130012, PRC, and School of Mathematics, Georgia Institute of
Technology, Atlanta, GA 30332, USA}
\email{yi@math.gatech.edu,yingfei@ualberta.ca}

\thanks {The first author was partially supported by NSFC grant 11225105. The second author was partially supported by NSFC
Innovation grant 10421101. The third author was partially supported
by  NSFC grant 11271151. The fourth author was partially supported by NSF
grant DMS1109201, a Scholarship from Jilin University, and a faculty
development fund from University of Alberta. The work
was partially done while the authors were visiting each other's
Institutions.}



\begin{abstract}  In this paper, we
study limit behaviors of stationary measures of the Fokker-Planck
equations associated with a system of ordinary differential
equations  perturbed by a class of
  multiplicative including additive white noises. As the noises
are vanishing, various results on the invariance  and
concentration of the limit measures
 are obtained. In particular,  we show that if the
 noise perturbed systems admit a uniform Lyapunov function, then
the stationary measures form a relatively sequentially compact set
 whose weak$^*$-limits are invariant
measures of the unperturbed system concentrated on its global
attractor. In the case that the global attractor contains a strong
local attractor, we further show that there exists a  family of
admissible multiplicative noises with respect to which all limit
measures are actually concentrated on the local attractor; and on
the contrary, in the presence of a strong local repeller in the
global attractor, there exists a family of admissible multiplicative
noises with respect to which no limit measure can be concentrated on
the local repeller. Moreover, we show that if there is a strongly
repelling equilibrium in the global attractor, then  limit measures
with respect to typical families  of  multiplicative noises are
always concentrated away from the equilibrium. As applications of
these results, an example of stochastic Hopf bifurcation is
provided.

Our study is closely related to the problem of noise
stability of compact invariant sets and invariant measures of the unperturbed system.
\end{abstract}
\maketitle

\section{Introduction}

Regarded as a physical model, a dynamical system
generated from ordinary differential equations is often subject to
noise perturbations either from its surrounding environment or from
intrinsic uncertainties associated with the system. Analyzing the
impact of noise perturbations on the dynamics of the system then
becomes a fundamental  issue with respect to both modeling and
dynamics.

There have been many studies  toward this dynamics
issue using either a trajectory-based or a distribution-based
approach. The  trajectory-based approach is often adopted under the
framework of random dynamical systems, i.e., skew-product flows with
 ergodic measure-preserving base flows. By assuming vanishing noise
at a reference equilibrium,  noise perturbations of essential
dynamics of a dynamical system are studied under  the random
dynamical system framework with respect to problems such as noise
perturbations of invariant manifolds (\cite{AK,DLS,DLS04,W}), normal
forms (\cite{Ar,AI,AX,LiLu}), and stochastic bifurcations (see
\cite{Ar} and references therein). For  a system of ordinary
differential equations subject to  white noise perturbations
vanishing at a reference equilibrium, we refer the reader to
\cite{Ha} for some study of stochastic stability of the equilibrium
(see also \cite{L} for similar studies in infinite dimension).

With respect to general noise perturbations, the distribution-based
approach is useful and seemly necessary to adopt under both
frameworks of random dynamical systems and It\^o stochastic
differential equations. Due to its essential differences from
deterministic dynamical systems, much less is known in this
direction comparing with cases using the trajectory-based approach.
For some import pioneer works on noise perturbations  of dynamical
systems on a compact manifold from the viewpoint of distributions,
we refer the reader to \cite{H63,N} for stochastic stability of
flows on a 2-torus or a periodic cycle, to \cite{FW} for stochastic
stability of equilibria and periodic cycles by introducing large
deviation theory, to \cite{CY,K,K74,You} for stochastic stability of
SRB measures, and to \cite{Zeeman} for some global stochastic
stability characterizations.

In this paper, we adopt the distribution-based approach to study the
impact of  white noises on basic dynamics of a system of ordinary
differential equations in an Euclidean space. More precisely, we
consider  a system of ordinary differential equations
\begin{equation}\label{ode}
\dot x=V(x),\qquad x\in \cal U\subset \R^n,
\end{equation}
where $\cal U$ is a connected open set which can be bounded,
unbounded, or the entire $\R^n$, and $V=(V^i)\in C(\cal U,\R^n)$. We
assume throughout the paper that \eqref{ode} generates a local flow
$\varphi^t$ on $\cal U$. The
 generality of  domain $\cal U$ does allow a wide range of applications because many physical models
 (e.g., those concerning populations and concentrations) are not necessarily defined in the entire $\R^n$.
Adding  general multiplicative (i.e., spatially non-homogenous)
including additive (i.e., spatially homogenous)  white noise
perturbations, we obtain the following It\^o stochastic differential
equations
\begin{equation}\label{sde}
\rmd {x}=V(x)\rmd t+ G(x)\rmd W, \qquad x\in \cal U\subset
\mathbb{R}^n,
\end{equation}
where $W$ is a standard $m$-dimensional Brownian motion for some
integer $m\ge n$, and $G=(g^{ij})_{n\times m}$ is a matrix-valued
function on $\cal U$, called {\it noise coefficient matrix}. For
generality, we assume that $g^{ij} \in W^{1,2{ p}}_{loc}(\cal U)$,
$i=1,2,\cdots,n$, $j=1,2,\cdots,m$, for some fixed constant ${p}
>n$.

The stochastic differential equations \eqref{sde} arise
naturally as a non-isolated physical system subject to noise
perturbations from its surrounding environments, in which the impact
of noises on dynamics is often physically measured in term of
distributions. They can also arise naturally from the study of a
large scale deterministic but seemly stochastic system, for instance
a so-called mesoscopic system which is partially structured but
contains intrinsic uncertainties in a fast time scale due to high
complexity, large degree of freedom, lack of full knowledge of
mechanisms, the need for organizing a large amount of data, etc.
Under some exponential mixing assumptions on the fast dynamics, such
a mesoscopic system can have a stochastic reduction of the form
\eqref{sde} over any finite time interval in which $V$ represents
the structured field and $G$ buries  all dynamical uncertainties
(see e.g., \cite{Just,Mackay}). It has been argued for a mesoscopic
system that in the case of sufficiently high uncertainty,
trajectory-based approach using either deterministic or random
dynamics modeling would not provide much information to its
dynamical description. Instead, a distribution-based approach using
stochastic differential equations like \eqref{sde} is necessary to
adopt in order to synthesize the typical patterns of dynamics (see
\cite{Qian} and references therein).

An important distribution-based approach for studying diffusion
process generated by  \eqref{sde} is to use its associated {\em
Fokker-Planck equation} (also called {\em Kolmogorov forward
equation})
\begin{equation}\label{fp}
\left\{\begin{array}{l}
\displaystyle\frac{\partial u(x,t)}{\partial t}=L_Au(x,t),\quad x\in \cal U,\, t>0, \\\
u(x,t)\ge 0, \quad \int_{\cal U}u(x,t)\rmd x=1,
\end{array}\right.
\end{equation}
where $A= (a^{ij})=\frac{GG^\top}2$, called the {\em diffusion
matrix}, and $L_A$ is the Fokker-Planck operator defined as
\[
L_Ag(x)= \partial^2_{ij}(a^{ij}(x)g(x))-
\partial_i(V^i(x)g(x)),\qquad g\in C^2(\cal U).
\]
We note that $a^{ij}\in W^{1,{p}}_{loc}(\cal U)$,
$i,j=1,2,\cdots,n$. It is well-known that if the stochastic
differential equation \eqref{sde} generates a (local) diffusion
process in $\cal U$ (e.g., when both $V$ and $G$ are locally
Lipschitz in $\cal U$), then its
 transition probability density function,
 if exists,  is actually a (local) fundamental solution of the
Fokker-Planck equation  \eqref{fp}.

In the above and also through the rest of the paper, we  use short
notions
% $u_t=\frac{\partial u}{\partial t}$,
$\partial_i =\frac{\partial }{\partial x_i}$,
$\partial^2_{ij}=\frac{\partial^2 }{\partial x_i
\partial x_j}$, and we also adopt the usual summation convention on $i,j=1,2,\cdots,n$
whenever applicable.

Long time behaviors of solutions of the Fokker-Planck equation
\eqref{fp} is governed by  the stationary Fokker-Planck equation
\begin{equation}\label{p2fk}
\left\{\begin{array}{l} L_A u =
\partial^2_{ij}(a^{ij}u)-\partial_i(V^i u)=0,\\
u(x)\ge 0, \quad \int_{\cal U}u(x)\rmd x=1,
\end{array}\right.
\end{equation}
 which, in the
weak form, becomes
\begin{equation}\label{j3} \left\{\begin{array}{l}\di_{\cal
U}{\cal L}_Af(x) u(x) \rmd x=0,\quad\quad
\hbox{for all } f\in C_0^\infty(\cal U),\\
u(x)\ge 0, \quad \int_{\cal U}u(x)\rmd x=1,
\end{array}\right.
\end{equation}
where $C_0^\infty(\cal U)$ denotes the space of $C^\infty$ functions
on $\cal U$ with compact supports and
\[
{\cal L}_A =a^{ij}\partial^2_{ij}  +V^i\partial_i
\]
is the adjoint Fokker-Planck operator corresponding to $A$.
Solutions of \eqref{j3} are called {\em weak stationary solutions of
\eqref{fp}} or {\em stationary solutions corresponding to ${\cal
L}_A$}. More generally, one considers a measure-valued stationary
solution $\mu_A$ of the Fokker-Planck equation \eqref{fp}, called a
{\em stationary measure of the Fokker-Planck equation \eqref{fp}} or
a {\em stationary measure corresponding to ${\cal L}_A$}, which is a
Borel probability measure satisfying
\begin{eqnarray}
%&&V^i\in L^1_{loc}(\cal U,\mu_A),\quad i=1,2,\cdots,n,\,\hbox{and},\label{Vmu}\\
%&&
 \di_{\cal U}{\cal L}_Af(x) \rmd\mu_A(x)=0,\quad\quad \hbox{for all  }
f\in C_0^\infty(\cal U).\label{j2}
\end{eqnarray}
If a stationary measure $\mu_A$ is regular, i.e.,
$\rmd\mu_A(x)=u_A(x)\rmd x$ for some density function $u_A\in C(\cal
U)$, then it is clear that $u_A$ is necessarily a weak stationary
solution of \eqref{fp}, i.e, it satisfies \eqref{j3}. Conversely,
 according to the regularity theorem in
\cite{BKR}, if $(a^{ij})$ is everywhere positive definite in $\cal
U$, then any stationary measure corresponding to ${\cal L}_A$  must
be regular with positive density function lying in $
W^{1,{p}}_{loc}(\cal U)$.  In the case that \eqref{sde} generates a
diffusion process on $\cal U$,  it is well-known that any invariant
measure of the diffusion process is necessarily a stationary measure
of the Fokker-Planck equation \eqref{fp},  but the converse need not
be true. However, under some mild conditions a stationary measure of
the Fokker-Planck equation \eqref{fp} is always a sub-invariant
measure of some generalized diffusion process (see \cite{BKR09,BRS1,
BRS} for  discussions in this regard in particular with respect to
the uniqueness of stationary measures and their invariance). In this
sense, stationary measures of \eqref{fp} may be regarded as
generalizations of invariant measures of a classical diffusion
process.

The existence of stationary measures of Fokker-Planck equations  in
$\R^n$ has been extensively investigated (see e.g.,
\cite{ABG},\cite{BKR}-\cite{BRS2},\cite{Ha},\cite{Ver97} and
references therein). In our recent work \cite{JM1}, such existence
is  investigated for a general domain  under certain relaxed
Lyapunov conditions. In addition, results concerning non-existence
of stationary measures of Fokker-Planck equations in a general
domain are also obtained in our work \cite{JM2} under some
anti-Lyapunov conditions, which, together with the existence results
in \cite{JM1}, lead to both sufficient and necessary conditions for
the existence of stationary measures of Fokker-Planck equations.

While the non-existence of stationary measures of Fokker-Planck
equations associated with a family of noise perturbations reflects a
strong stochastic instability of the unperturbed deterministic
system with respect to these noises, stochastic stability and
instability of the deterministic system at a dynamics level can
often occur with respect to noise families for which stationary
measures do exist for the corresponding Fokker-Planck equations. In
order to study  the impact of noises on such stochastic stability of
the deterministic system  when it generates a local flow, a
fundamental problem is to  classify basic dynamics subjects like
compact invariant sets and invariant measures of the local  flow
that can ``survive" from a given family of noise perturbations. This
is in fact our main motivation for the present study.
%We note that
%there are natural conditions which guarantee the flow extension of
%the local flow in a relevant sub-domain on which invariant sets and
%invariant measures are meaningful (see Section 6 for details).

To be more precise, for a fixed drift field $V\in C(\cal U,\R^n)$,
we consider noise coefficients matrices lying in the class
\begin{eqnarray*}
\tilde{\cal G}=
 \{ G=(g^{ij}): {\rm Rank}( G)\equiv n,\, g^{ij} \in W^{1,2{p}}_{loc}(\cal
 U),\,i=1,2,\cdots,n,\,j=1,2,\cdots,m\}
 \end{eqnarray*}
for some fixed ${p}>n$.  The class $ \tilde{\cal G}$ gives rise to
the following class of  diffusion matrices:
\begin{equation}\label{admissible}
\tilde{\cal A}=\{A=(a^{ij})\in W^{1,{p}}_{loc}(\cal U,
GL(n,\mathbb{R})): \hbox{ $A=\frac{GG^\top}{2}$\hbox{ for some
}$G\in \tilde{\cal G}$}\}.
\end{equation}
To consider small noise perturbations, we will pay particular
attention to the so-called  {\em null family} (resp. {\em bounded
null family}) $\cal A=\{A_\alpha\}\subset \tilde{\cal A}$, i.e., a
directed net of $\tilde{\cal A}$ which converges to $0$ - the zero
matrix, under the topology of $W^{1,{p}}$-convergence on any compact
subsets of $\cal U$ (resp. $L^\infty$-convergence on $\cal U$). Then
with respect to the system \eqref{ode} and a given null family $\cal
A=\{A_\alpha\}\subset \tilde{\cal A}$, our study amounts to the
characterization of behaviors of {\em $\mathcal{A}$-limit measures},
i.e., sequential limit points, as $A_\alpha\to 0$, of stationary
measures $\{\mu_{\alpha}\}$ corresponding to $\{{\cal
L}_{A_\alpha}\}$ under the weak$^*$-topology. Our particular
attention will be paid to issues such as invariance and
concentration of $\cal A$-limit measures, stochastic stability of
compact invariant sets under uniform Lyapunov conditions,  and the
role played by the multiplicative noises $\cal A$ to the
stabilization of a strong local attractor or de-stabilization of a
strong local repeller.

Our main results of the paper are as follows.
\medskip

\noindent{\bf Theorem~A.} {\em Let $\cal A=\{A_\alpha\}\subset
\tilde{\cal A}$ be a null family. Then the following holds.
\begin{itemize}
\item[{\rm a)}] {\rm (Invariance of limit measures)} If $V\in C^1(\cal
U,\R^n)$ and  $\varphi^t$ is a flow on $\cal U$, then any
$\mathcal{A}$-limit measure must be an invariant measure of
$\varphi^t$.

\item[{\rm b)}] {\rm (Stochastic LaSalle  invariance principle)} If \eqref{ode} admits  an
entire weak Lyapunov {\rm(}resp. anti-Lyapunov{\rm )} function
$U_0$, then any $\mathcal{A}$-limit measure with compact support  is
concentrated on the set $S_0=\{x\in \cal U: V(x)\cdot \nabla
U_0(x)=0\}$.
\item[{\rm c)}] {\rm (Local concentration of limit measures)} If $V\in C^1(\cal U,\R^n)$ and $\cal
E_0$ is either a strong local attractor or a strong local repeller
of $\varphi^t$, then there is a neighborhood $\cal W_0$ of $\cal
E_0$ in $\cal U$ such that $\mu(\cal W_0\setminus \cal E_0)=0$ for
any $\mathcal{A}$-limit measure $\mu$.
\end{itemize}
} \medskip

 Thus,  when both $V$ and
$A_\alpha'$s are smooth in $\cal U$, it follows from part a) of
Theorem~A that the stationary measures $\{\mu_{\alpha}\}$
corresponding to $\{{\cal L}_{A_\alpha}\}$, by having smooth density
functions, can be regarded as smoothers of their limit invariant
measures.

We note that parts b), c) of Theorem~A do not require  $\varphi^t$
be a flow on $\cal U$, in which case there is no guarantee that an
$\cal A$-limit measure is an invariant measure of $\varphi^t$.
However, in the case that $\varphi^t$ is only a semiflow on $\cal
U$,  the invariance  of an $\cal A$-limit measure can be also shown
under certain conditions (see Theorem~\ref{tight0-1} c) and Theorem
B below).

Given a null family $\cal A=\{A_\alpha\}\subset \tilde{\cal A}$, if
there exists a uniform Lyapunov function in $\cal U$  with respect
to the family $\{{\cal L}_{A_\alpha}\}$ (see Section 2), then there
must exist a stationary measure $\mu_{\alpha}$ corresponding to each
${\cal L}_{A_\alpha}$ (Proposition \ref{exist}) and $\varphi^t$ is
necessarily a dissipative semiflow on $\cal U$
(Proposition~\ref{UNIF}) whose global attractor necessarily admits a
flow extension (Proposition~\ref{omega-set1}).
\medskip

\noindent{\bf Theorem~B.} {\em Let $\cal A=\{A_\alpha\}\subset
\tilde{\cal A}$ be a null family.  Assume that there is a uniform
Lyapunov function in $\cal U$ with respect to the family $\{{\cal
L}_{A_\alpha}\}$ and denote by $\{\mu_{\alpha}\}$ the set of all
stationary measures corresponding to $\{{\cal L}_{A_\alpha}\}$. Then
the following holds.
\begin{itemize}
\item[{\rm a)}] {\rm (Tightness of stationary measures)} $\{\mu_{\alpha}\}$
is $\cal A$-sequentially null compact in $M(\cal U)$ - the space of
Borel probability measures on $\cal U$, i.e., for any sequence
$\{A_\ell\}_{\ell=1}^\infty\subset \mathcal{A}$ with
$A_\ell\rightarrow 0$, $\{\mu_{\ell}\}_{\ell=1}^\infty$ is
relatively  compact in $M(\mathcal{U})$.
\item[{\rm b)}] {\rm (Global concentration of limit measures)} If $V\in C^1(\cal U,\R^n)$,
then any $\mathcal{A}$-limit measure is an invariant measure of
$\varphi^t$  concentrated on the global attractor of $\varphi^t$.

\end{itemize}
}
\medskip

Concentration of $\mathcal{A}$-limit measures  on a compact
invariant set of $\varphi^t$ is a primary feature of stochastic
stability of the invariant set  with respect to the noise family
$\cal A$ - the so-called {\em $\cal A$-stability} which we will
define in Section~2. For instance, Theorem~B  actually imply that
the global attractor of $\varphi^t$ is $\cal A$-stable
if $\cal A$ is a so-called invariant null family.
%In fact,
%{\color{red} with respect to an null family $\cal A$}, all $\cal
%A$-limit measures characterize a unique minimal $\cal A$-stable set
%${\cal J}_{\cal A}$ contained in the global attractor, i.e., ${\cal
%J}_{\cal A}$ contains no proper $\cal A$-stable compact invariant
%subset of $\varphi^t$.
Such $\cal A$-stability
  is also closely related to  the so-called $\cal A$-stability of an invariant measure of $\varphi^t$. We refer the
reader to Sections~2,~3 for more discussions in these regards.

Multiplicative noises actually play important roles in stabilizing a
local attractor  or de-stabilizing a local repeller of $\varphi^t$.
To analyze such roles played by multiplicative noises, we assume
that \eqref{ode} admits a Lyapunov function in $\cal U$ whose second
derivatives are bounded. This Lyapunov function then becomes a
uniform Lyapunov function with respect to a bounded null family in
$\tilde{\cal A}$ (Proposition~\ref{UNIF-S}) for which Theorem~B is
applicable.
%, for which
%multiplicative noises play important roles within the class of the
%so-called bounded or normal null family (see Section~2).
\medskip

\noindent{\bf Theorem~C.} {\em Assume  $V\in C^1(\cal U,\R^n)$ and
that \eqref{ode} admits a Lyapunov function in $\cal U$ whose second
derivatives are bounded. Then the following holds.
\begin{itemize}
\item[{\rm a)}] {\rm  (Noise stabilization of local attractors)} If the global attractor of $\varphi^t$
contains a strong local attractor, then there exists a normal null
family $\cal A=\{A_\alpha\}\subset \tilde{\cal A}$  such that
stationary measures  corresponding to $\{{\cal L}_{A_\alpha}\}$
exist, form an $\cal A$-sequentially null compact set in $M(\cal
U)$, and any $\mathcal{A}$-limit measure is an invariant measure of
$\varphi^t$ concentrated on the local attractor.
\item[{\rm b)}] {\rm (Noise de-stabilization  of local repellers)} If the global attractor of $\varphi^t$
contains a strong local  repeller, then there exists a normal null
family $\cal A=\{A_\alpha\}\subset \tilde{\cal A}$  such that
stationary measures  corresponding to $\{{\cal L}_{A_\alpha}\}$
exist, form an $\cal A$-sequentially null compact set in $M(\cal
U)$, but all $\mathcal{A}$-limit measures are concentrated away from
the local repeller.
\item[{\rm c)}] {\rm (Noise instability of  repelling
equilibria)} If the global attractor of $\varphi^t$ contains a
strongly  repelling equilibrium, then with respect to any normal
null family  $\cal A=\{A_\alpha\}\subset \tilde{\cal A}$ on
$\mathcal{U}$, stationary measures  corresponding to $\{{\cal
L}_{A_\alpha}\}$ exist, form an $\cal A$-sequentially null compact
set in $M(\cal U)$, but all $\mathcal{A}$-limit measures are
concentrated away from the equilibrium.
\end{itemize}
}
\medskip

In the above theorem, a normal null family is a more restricted
bounded null family defined in Section~2. For definitions of a
strong local attractor or repeller and strongly repelling
equilibrium of \eqref{ode}, we refer the reader to Section~6.

The proof of Theorem~C actually contains explicit conditions on the
noise family $\cal A$ under which a) or b) above holds. For
instance, in the case of a), these conditions actually quantify how
strong the noises should be away from the local attractor in order
to achieve the desired stabilization. We refer the reader to
Section~4 for details.

We note that Theorem~C  can actually be re-stated as such that the
strong local attractor in a) is $\cal A$-stable with respect to some
invariant, bounded null family $\cal A$, the strong local repeller
in b) is strongly $\cal A$-unstable with respect to some invariant,
bounded null family $\cal A$, and the repelling equilibrium in c) is
strongly $\cal A$-unstable with respect to any invariant, normal
null family $\cal A$.

This paper is organized as follows. Section 2 is a preliminary
section in which we mainly review some fundamental properties of
stationary measures of  Fokker-Planck equations associated with
\eqref{sde} and introduce various notions of stochastic stability
for compact invariant sets and invariant measures of $\varphi^t$. In
Section~3, we study behaviors and concentration of limit measures.
Possible local concentration of a limit measure on a strong local
attractor or repeller will  be studied and invariance or
semi-invariance of limit measures will be discussed based on some
new characterizations on invariant and semi-invariant measures of
$\varphi^t$. Theorems~A,~B will be proved in this section. In
Section~4, we prove Theorem~C by characterizing noise families that
stabilize a strong local attractor or de-stabilize  a strong local
repeller. In Section~5, we demonstrate applications of our main
results by considering an example of stochastic Hope bifurcation in
which the existence of a  stochastically stable cycle is observed.
Section 6 is an Appendix in which we  summarize some basic notions
and dynamical properties for a system of dissipative ordinary
differential equations.

We remark that if the system \eqref{ode} is defined on $\cal U\times
M$, where $M$ is a smooth,  compact manifold without boundary (e.g.
$M=\T^k$,  the $k$-torus), then one can modify the definitions of
uniform Lyapunov and  anti-Lyapunov  functions  by replacing the
domain $\cal U\subset \R^n$ with $\cal U\times M$ so that all
results in Sections~3,~4  hold with respect to such a generalized
domain.
%In addition, these results also allow the variation of drift
%field $V$ in a family $\cal V$ (e.g., equations like \eqref{sdee})
% the Lyapunov functions are  uniform with respect to  both $\cal
%V$ and $\cal A$.

Through the rest of the paper, for simplicity, we will use the same
symbol $|\,\cdot|$ to denote the absolute value of a number, the
norm of a vector, a matrix, and the Lebesgue measure of a set. For
any connected open set $\Omega\subset\cal U$ and any integer $0\le
M<\infty$ or $M=\infty$, we denote by $C_0^M(\Omega)$  the set of
functions in $C^M(\Omega)$ with compact supports in $\Omega$ and by
$C_b^M(\Omega)$ the set of functions with bounded derivatives up to
order $M$. When $M=0$, we denote $C^0_0(\Omega),C^0_b(\Omega)$
simply by $C_0(\Omega),C_b(\Omega)$ respectively.

\section{Preliminary}

In this  section, we will review some fundamental properties of
stationary measures of the Fokker-Planck equation associated with
\eqref{sde} from \cite{JM,JM1} and define various notions of
stochastic stability for compact invariant sets and invariant
measures. We will also  recall  a Harnack inequality to be used in
later sections.

\subsection{Lyapunov function and  stationary measures}
Let $A=(a^{ij})$  be a given everywhere positive definite, ${n\times
n}$ matrix-valued function on $\cal U$ such that $a^{ij}\in
W^{1,{p}}_{loc}(\cal U)$, $i,j=1,2,\cdots, n$. The matrix-velued
function corresponds to an adjoint Fokker-Planck operator:
\[
{\cal L}_{A}=a^{ij}\partial^2_{ij}  +V^i\partial_i
\]
which defines a  weak form of stationary Fokker-Planck equation
\eqref{j3}.

Let $\Omega$ be a connected open subset of $\cal U$.  We recall from
 \cite{JM,JM1} that a non-negative function $U\in C(\Omega)$ is a {\em
compact function} if i) $U(x)<\rho_M$, $x\in \Omega$; and ii)
$\lim_{x\to
\partial{\Omega}} U(x)=\rho_M$, where
  $\rho_M= \sup_{x\in \Omega} U(x)$ is called  the {\em
 essential upper bound of $U$}. For each $\rho\in  [0,\rho_M)$, we
 denote $\Omega_\rho=\{x\in \Omega: U(x)<\rho\}$
 as the $\rho$-sublevel set of $U$ and $U^{-1}(\rho)=\{x\in \Omega: U(x)=\rho\}$ as the
 $\rho$-level set of $U$. In the above, the notion $\partial\Omega$ and limit $x\to\partial\Omega$
 are defined through a unified
topology which identifies the extended Euclidean space
$\E^n=\R^n\cup
\partial \R^n$ with the closed unit ball
 $\bar \B^n=\B^n\cup \partial \B^n$ in $\R^n$ so that $\partial
 \R^n$, consisting of infinity elements  of all rays, is identified
 with
$\partial \B^n=\S^{n-1}$ (see \cite{JM,JM1} for details). Therefore,
when $\Omega=\R^n$, the limit $x\to\partial\R^n$  is simply
equivalent to $x\to\infty$.
\medskip

We also recall from \cite{JM}-\cite{JM2} the following notions of
Lyapunov-like and anti-Lyapunov-like functions.

 \begin{Definition}\label{Definition120} {\em A $C^2$ compact function
$U$  is called a {\em Lyapunov function} (resp. {\em  anti-Lyapunov
function}) {\em in $\Omega$  with respect to  ${\cal L}_A$},  if
there is a $\rho_m\in (0,\rho_M)$, called {\em essential lower bound
of $U$}, and a constant $\gamma>0$, called {\em Lyapunov constant}
(resp. {\em  anti-Lyapunov constant}) of $U$, such that
\begin{equation}\label{U3}
{\cal L}_A U(x)\leq -\gamma, \quad\quad\,\hbox{(resp.}\, \ge
\gamma{\rm)},\;\;\, x\in \tilde{\Omega}=\Omega\setminus
\bar\Omega_{\rho_m},
\end{equation}
where $\tilde{\Omega}$ is called {\em essential domain of $U$}. If
$\gamma=0$ in the above, then $U$ is referred to as a {\em weak
Lyapunov function} (resp. {\em weak anti-Lyapunov function}) in
$\Omega$  with respect to  ${\cal L}_A$. }
\end{Definition}

\begin{Proposition}\label{exist} {\rm (}\cite[Theorem A]{JM1}{\rm )}
If  there is a Lyapunov function in $\cal U$ with respect to ${\cal
L}_{A}$, then  a stationary measure $\mu_{A}$ corresponding to
${\cal L}_{A}$ exists and is regular in the sense that  ${\rm
d}\mu_{A}( x)=u_{A}(x){\rmd} x$ for some weak stationary solution
$u_{A}\in W^{1,{p}}_{loc}(\cal U)$ corresponding to ${\cal L}_{A}$.
Moreover, if the Lyapunov function is unbounded, then the stationary
measure is unique.
\end{Proposition}

Below, we recall from \cite{JM,JM1} some  measure estimates which
are derived based on the level set method introduced in these works.
In fact, these estimates only require the weaker regularity
condition on the drift field $V$ that $V^i\in L^{{ p}}_{\rm
loc}(\cal U)$, $i=1,\cdots,n$.

\begin{Proposition}\label{unif}
Assume that there is  a Lyapunov function $U$ in
 $\Omega$ with respect to ${\cal L}_A$, with Lyapunov constant $\gamma$,
 essential lower bound $\rho_m$ and upper bound $\rho_M$. Let
$\Omega_\rho$  denote the $\rho$-sublevel set of $U$ for each
$\rho\in [\rho_m,\rho_M)$. Then the following holds for any
stationary measure $\mu$ corresponding to ${\cal L}_A$:
\begin{itemize}
\item[{\rm a)}]{\rm(}\cite[Lemma 4.1]{JM1}{\rm)} For  any $\rho_0\in
(\rho_m,\rho_M)$, there exists a constant $C_{\rho_m,\rho_0}>0$
depending only on $\rho_m,\rho_0$ such that
\[
 \mu({\cal U}\setminus
\Omega_{\rho_0})\le \gamma^{-1}
C_{\rho_m,\rho_0}|A|_{C(\Omega_{\rho_0}\setminus \Omega_{\rho_m})}
|\nabla U|_{C(\Omega_{\rho_0}\setminus
\Omega_{\rho_m})}^2\mu(\Omega_{\rho_0}\setminus \Omega_{\rho_m}).
\]

\item[{\rm b)}] {\rm(}\cite[Theorem A b)]{JM}{\rm)}  If, in addition,
\begin{eqnarray}
&& \nabla U(x)\ne 0,\,\; \quad\forall x\in U^{-1}(\rho)\;\, \, {\rm
for }\;\, a.e. \;\, \rho\in [\rho_m,\rho_M),\label{t1} \\
 && a^{ij}(x)\partial_i
U(x)\partial_j U(x)\le H(\rho),\qquad x\in
\partial\Omega_\rho,\;\;\rho\in [\rho_m,\rho_M)\label{t2}
\end{eqnarray}
for some  non-negative measurable function $H$ defined on
$[\rho_m,\rho_M)$, then
$$\mu({\cal U}\setminus  \Omega_{\rho})\le
\rme^{- \gamma\int_{\rho_m}^{\rho} \frac {1}{H(t)}\rmd t},\qquad
\rho\in [\rho_m,\rho_M).
$$
\end{itemize}
\end{Proposition}
\medskip

\begin{Proposition} {\rm(}\cite[Theorem B a)]{JM}{\rm)} \label{aunif} Assume that there is  an
anti-Lyapunov function $U$ in
 $\Omega$ with respect to ${\cal L}_A$, with anti-Lyapunov constant $\gamma$,
 essential lower bound $\rho_m$ and upper bound $\rho_M$. Let
$\Omega_\rho$  denote the $\rho$-sublevel set of $U$ for each
$\rho\in [\rho_m,\rho_M)$. If $U$ satisfies \eqref{t1}, \eqref{t2}
 with respect to
 a non-negative measurable function $H$ defined on
$[\rho_m, \rho_M)$,  then the following holds for any stationary
measure $\mu$ corresponding to ${\cal L}_A$:
$$\mu(\Omega_{\rho}\setminus \Omega^*_{\rho_m})\geq
\mu(\Omega_{\rho_0}\setminus  \Omega^*_{\rho_m})\rme^{\gamma
\int_{\rho_0}^\rho \frac 1{H(t)}\rmd t},\quad\quad \rho\in
(\rho_0,\rho_M),$$ where $\Omega^*_{\rho_m}=\Omega_{\rho_m}\cup
U^{-1}(\rho_m)=\{x\in \cal U: U(x)\le \rho_m\}$.
\end{Proposition}

%\medskip

%\begin{Proposition} {\rm(}\cite[Theorem B b)]{JM}{\rm)} \label{weak-anti}
%Assume  that there is a  weak anti-Lyapunov function $U$ in $\Omega$
%with respect to ${\cal L}_A$  which satisfies
%\[
%H_1(\rho)\le a^{ij}(x)\partial_i U(x)\partial_j U(x)\le
%H_2(\rho),\qquad x\in U^{-1}(\rho),\;\;\rho\in [\rho_m,\rho_M)
%\]
%for two positive, measurable  functions $ H_1\leq H_2$ defined on
%$[\rho_m,\rho_M)$ with $H_1$ being continuous, where $\rho_m$ and
%$\rho_M$ are the essential lower and upper bounds of $U$,
%respectively. Then the following holds for any stationary measure
%$\mu$ corresponding to  ${\cal L}_A$  and any $\rho_0\in
%(\rho_m,\rho_M)$:
%\[
%\mu(\Omega_\rho\setminus \Omega_{\rho_m})\ge
%\mu(\Omega_{\rho_0}\setminus \Omega_{\rho_m})
%\rme^{\int_{\rho_0}^\rho \frac 1{\tilde H(t)}\rmd t},\qquad \rho\in
%[\rho_0,\rho_M),
%\]
%where
%\[
%\tilde H(\rho)=H_2(\rho)\int_{\rho_m}^{\rho} H_1^{-1}(s)\rmd s,
%\qquad\quad \rho\in [\rho_0,\rho_M).
%\]
%\end{Proposition}

%\begin{Definition}\label{tightness0} {\em A subset $\mathcal{M}\subset M(\Omega)$ is said to be {\em tight} if for any
%$\epsilon>0$ there exists a compact subset $K_\epsilon \subset
%\Omega$ such that $\mu(\Omega \setminus K_\epsilon)<\epsilon$ for
%all $\mu\in \mathcal{M}$. }
%\end{Definition}
%\medskip

%\begin{Theorem}\label{Prokh} {\rm (Prokhorov's Theorem, e.g. \cite[Theorem 5.1]{Bi})}   If a subset
%$\mathcal{M}\subset M(\Omega)$ is tight, then it is relatively
%sequentially compact in $M(\Omega)$.
%\end{Theorem}
%\medskip

% We note that if $\Omega$ is
%compact, then any subset of $M(\Omega)$ is tight.

\subsection{Null family} Consider the  class $\tilde{\cal A}$ of  diffusion matrices
  defined in \eqref{admissible}.

\begin{Definition}\label{nullD} {\em 1) A {\em null family} (resp. {\em bounded null family})
$\cal A=\{A_\alpha\}\subset\tilde{\cal A}$  is a directed net of
$\tilde{\cal A}$  which converges to $0$ - the zero matrix, under
the topology of $W^{1,{p}}$-convergence on any compact subsets of
$\cal U$ (resp. $L^\infty$-convergence on $\cal U$).

2) A bounded null family $\{A_\alpha =(a^{ij}_\alpha)\} \subset
\tilde{\cal A}$  is said to be a {\em normal null family} if for any
pre-compact open subset $\Omega$ of $\cal U$,
\begin{equation}\label{normal}
\sup_\alpha\frac{\Lambda_\alpha(\Omega)}{\lambda_\alpha(\Omega)}<\infty,
\end{equation}
 where, for each
$\alpha$,
\[
\Lambda_\alpha(\Omega)=\sup_{x\in \Omega}\Lambda_\alpha(x),\quad
\quad
 \lambda_\alpha(\Omega)=\inf_{x\in
\Omega}\lambda_\alpha(x)
\]
with $\Lambda_\alpha(x)=\sqrt{\sum_{i,j} |a^{ij}_\alpha(x)|^2}$ and
$\lambda_\alpha(x)=\inf_{\xi\in\R^n\setminus\{0\}} \frac{\xi^\top
A_\alpha (x) \xi}{|\xi|^2}$ for each $x\in \Omega$. }
\end{Definition}

We note that in the above   $\lambda_\alpha(x)$ is just the smallest
eigenvalue of $A_\alpha(x)$ and $\Lambda_\alpha(x)$ is strictly
bigger than the largest eigenvalue of $A_\alpha(x)$.

On $\tilde{\cal A}$, the topology of $W^{1,{p}}$-convergence on any
compact subsets of $\cal U$ is metrizable. Through  the rest of the
paper,  we denote an equivalent metric generating this topology by
$d$ and denote the $L^\infty$ norm on the subspace of bounded
elements of $\tilde{\cal A}$ by $|\cdot |$ for short.
\medskip

\begin{Remark}~\label{null} {\em  By Sobolev embedding $W^{1,\bar
p}(\Omega)\hookrightarrow C(\bar\Omega)$ for any pre-compact open
set $\Omega\subset\cal U$  with $C^1$ boundary $\partial\Omega$, if
$\cal A=\{A_\alpha\}\subset\tilde{\cal A}$ is a null family, then
$A_\alpha\to 0$  uniformly   on any compact subsets of $\cal U$.
This is because $\cal U$ can be approximated from inside by
pre-compact open subsets with smooth boundary. }
\end{Remark}

\begin{Definition}\label{def2.3} {\em   Let $\cal A\subset \tilde{\cal
A}$ be a null family.
\begin{itemize}
\item[{\rm 1)}] $\cal A$ is said to be {\em invariant} if for each
$A=\frac{GG^\top}2\in \mathcal{A}$, $\cal U$ is invariant with
respect to the diffusion process corresponding to $G$ or $A$, i.e.,
with probability one, any solution of \eqref{sde} starting in $\cal
U$ remains in $\cal U$ for all positive time of existence.

\item[{\rm 2)}] $\cal A$ is said to be {\em admissible} if for each
$A=\frac{GG^\top}2\in \mathcal{A}$, there exists a stationary
measure corresponding to ${\cal L}_{A}$. \end{itemize} }
\end{Definition}

\begin{Remark}\label{well-post}  {\em Let $\cal A$ be a
null family.

 1) With our assumptions on $V$ and $G\in \tilde{\cal G}$, local
existence of solutions of \eqref{sde} is guaranteed. The invariance
of a null family $\cal A$ only requires that such solutions
corresponding to each $A=\frac{GG^\top}2\in \cal A$ do not escape
$\cal U$ when time evolves, and local uniqueness of them is not even
required.

 2) When $\cal U=\R^n$, it is obvious that any null family $\cal A$
is automatically invariant. When $\cal U$ admits a boundary in
$\R^n$, the  invariance property of a null family $\cal A$ naturally
leads to suitable boundary conditions (on $V$ or $A\in \cal A$ or
both) of reflection, stopping, or general Wentzell types {\rm (}see
\cite{F1,F2,Went}{\rm ) for \eqref{sde} or \eqref{fp}. For instance,
if  \eqref{ode} admits a Lyapunov function in $\cal U$ and each
$A\in \cal A$ satisfies $A(x)\to 0$  as $x\to
\partial{\cal U}$, then it is not hard to see that $\cal A$ is
invariant.

3) For a null family $\cal A$ to be both invariant and admissible,
certain  uniform Lyapunov-like conditions (see below) are usually
needed. When $\cal U$ is bounded, it is true by the theory of
elliptic equations that if each $A\in\cal A$ is positive definite on
$\bar{\cal U}$ then there exists a stationary measure in $\cal U$
corresponding to each ${\cal L}_{A}$
 (\cite{BR01,JM1}). But  such a null family will fail to be
 invariant. To ensure the admissibility of a null family $\cal A$ in the case of degeneracy of
$A\in \cal A$ on $\partial{\cal U}$, uniform Lyapunov-like
conditions are still necessary (see \cite{JM1}). } }
\end{Remark}

%Let   $U$ be a $C^2$ compact function in $\cal U$ with essential
%upper bound  $\rho_M$.

%We recall from \cite{JM1,JM2} the following notions of
%Lyapunov-like and anti-Lyapunov-like functions.

 \begin{Definition}\label{Definition1200} {\em Let ${\cal
A}=\{A_\alpha\}\subset \tilde{\cal A}$ be a null family  and
$\Omega\subset \cal U$ be a connected open set.
 A $C^2$ compact function
$U$ is a {\em uniform Lyapunov function} (resp. {\em uniform
anti-Lyapunov function}) {\em  in $\Omega$ with respect to $\cal A$
or $\{{\cal L}_{A_\alpha}\}=:\{{\cal L}_A\}_{A\in\cal A}$ }  if it
is a Lyapunov function (resp. anti-Lyapunov function) in $\Omega$
with respect to each ${\cal L}_{A}$,  $A\in \cal A$, and the
essential lower bound $\rho_m$ and Lyapunov (resp. anti-Lyapunov)
constant $\gamma$, of $U$, are
 independent of $A\in \cal A$.
 }
 \end{Definition}

%  Recall from
%\cite{JM1,JM2} that a {\em uniform Lyapunov function} (resp. {\em
%uniform anti-Lyapunov function})  in $\Omega$ with respect to $\cal
%A$ or $\{{\cal L}_{A_\alpha}\}=:\{L_A^*\}_{A\in\cal A}$ is a compact
%function $U\in C^2(\Omega)$ such that, for all $A\in\cal A$,
%\begin{equation}\label{U3}
%L_A^* U(x)\leq -\gamma \quad\quad\,\hbox{(resp.}\, \ge
%\gamma{\rm)},\;\;\, x\in \tilde{\Omega}=\Omega\setminus
%\bar\Omega_{\rho_m}
%\end{equation}
%for some  constant $\gamma>0$,  called {\em uniform Lyapunov
%constant} (resp. {\em uniform anti-Lyapunov constant}), where
%$\tilde{\Omega}$ is  called {\em essential domain of $U$},
%$\rho_m\in (0,\rho_M)$ is a constant,  called  {\em essential lower
%bound of $U$}, and $\rho_M=\sup_{x\in\Omega} U(x)$ is the {\em
%essential upper bound of $U$}.
%If $\cal A =\{A\}$ consists of one
%element, then the above $U$ is a {\em Lyapunov function} (resp. {\em
%an anti-Lyapunov function})  in $\Omega$ with respect to $A$ or
%${\cal L}_A$; correspondingly, $\gamma$ is {\em Lyapunov constant} (resp.
%{\em anti-Lyapunov constant}), $\tilde{\Omega}$ is  {\em essential
%domain of $U$}, $\rho_m$ is {\em essential lower bound of $U$}, and
%$\rho_M$ is the {\em essential upper bound of $U$}. If
%\begin{equation*}
%L_A^* U(x)\leq 0 \quad\quad\,\hbox{(resp.}\, \ge 0{\rm)},\;\;\, x\in
%\tilde{\Omega}=\Omega\setminus \bar\Omega_{\rho_m},
%\end{equation*}
%then $U$ is called a {\em weak Lyapunov function} (resp. {\em weak
%anti-Lyapunov function})  in $\Omega$ with respect to $A$ or
%${\cal L}_A$.

\medskip

\begin{Remark}\label{rk2.3} {\em  1) It follows immediately from
Proposition~\ref{exist} that if there is a uniform Lyapunov function
in $\cal U$ with respect to a null family $\cal A$, then $\cal A$ is
admissible.

2)  Consider  $A_\epsilon=\epsilon A$, $0<\epsilon\le 1 $,  for a
fixed $A=(a^{ij})\in \tilde{\cal A}$. Then it is clear that
$\{A_\epsilon\}$ is a null family, and in fact a normal
null family if $A$ is uniformly positive definite on $\cal U$.
We note that a uniform Lyapunov function can be easily
obtained with respect to such a null family. Assume that $U$ is a
Lyapunov function in $\Omega\subset\cal U$ with respect to the
operator ${\cal L}_A$ with Lyapunov constant $\gamma$. If $U$ is
quasi-convex (i.e., the Hessian matrix $D^2U$ is positive
semi-definite) near $\partial\Omega$, then
\[
{\cal L}_{A_\epsilon}=\epsilon a^{ij}(x)\partial^2_{ij}U(x)
+V^i(x)\partial_i U(x)\le a^{ij}(x)\partial^2_{ij}U(x)
+V^i(x)\partial_i U(x)\le -\gamma
\]
for all $x\in \Omega$ sufficiently close to $\partial\Omega$, i.e.,
$U$ becomes a uniform Lyapunov function in $\Omega$ with respect to
the family $\{{\cal L}_{A_\epsilon}\}$. }
\end{Remark}

 A uniform Lyapunov (resp. anti-Lyapunov) function with respect to
a null family is actually a Lyapunov (resp. anti-Lyapunov) function
of the system  \eqref{ode}.

\begin{Proposition}\label{UNIF} Suppose that $U$ is a  uniform Lyapunov {\rm (}resp. anti-Lyapunov{\rm )}
function in a connected open set $\Omega\subset \cal U$  with
respect to a null family $\cal A=\{A_\alpha\}\subset \tilde{\cal
A}$. Then $U$ must be a Lyapunov  {\rm (}resp. anti-Lyapunov{\rm )}
function of \eqref{ode} in $\Omega$. Consequently, \eqref{ode}
generates a positive {\rm (}resp. negative{\rm)} semiflow which is
dissipative {\rm (}resp. anti-dissipative{\rm )} in $\Omega$.
\end{Proposition}
\begin{proof} We only prove the case when $U$ is a uniform Lyapunov
function in $\Omega$ with respect to $\{{\cal L}_{A_\alpha}\}$.

 Denote $\rho_m$, respectively $\rho_M$, as the essential
lower, respectively upper, bound of $U$, $\gamma$ as a uniform
Lyapunov constant of $U$,  and $\Omega_{\rho}$ as the
$\rho$-sublevel set of $U$ for each $\rho\in [\rho_m,\rho_M)$. Then
\[
{\cal L}_{A_\alpha}U(x)=a^{ij}_\alpha(x)\partial^2_{ij}U(x)
+V^i(x)\partial_i U(x)\le -\gamma,\qquad x\in \Omega\setminus
\bar\Omega_{\rho_m}.
\]
By taking limit $A_\alpha\to 0$ in the above, we have
\begin{equation}\label{VU}
V(x)\cdot \nabla U(x)\le -\gamma,\qquad x\in \Omega\setminus
\bar\Omega_{\rho_m},
\end{equation}
i.e.,  $U$ is a Lyapunov function of \eqref{ode} in $\Omega$. It
follows from Proposition~\ref{dissip} that \eqref{ode} generates a
positive  semiflow which is dissipative  in $\Omega$.
\end{proof}

Conversely, with respect to a bounded null family, a uniform Lyapunov function
can be naturally obtained from a Lyapunov function of \eqref{ode}.
\medskip

\begin{Proposition}\label{UNIF-S} Let $\Omega\subset \cal U$ be a connected open
set and suppose that \eqref{ode} admits a $C^2$ Lyapunov  {\rm
(}resp. anti-Lyapunov{\rm )}  function $U$ whose second derivatives
are bounded in $\Omega$.  Then $U$  is a uniform Lyapunov {\rm
(}resp. anti-Lyapunov{\rm )} function in $\Omega$ with respect to
any bounded null family $\cal A=\{A_\alpha\}\subset \tilde{\cal A}$
 with $|A_\alpha|\ll 1$.
\end{Proposition}

\begin{proof} We only consider the case that $U$ is a  Lyapunov  function of $\varphi^t$ in
$\Omega$, i.e.,
\[
V(x)\cdot \nabla U(x)\le -\gamma,\qquad\qquad\; x\in \tilde{\Omega},
\]
where $\tilde{\Omega}$, $\gamma$ are essential domain, Lyapunov
constant, of $U$, respectively. Let  $\cal
A=\{A_\alpha\}=\{(a^{ij}_\alpha)\}\subset \tilde {\cal A}$ be a
bounded null family. Then with $|A_\alpha|\ll 1$,
\[
{\cal L}_{A_\alpha} U(x)= a^{ij}_\alpha\partial^2_{ij} U(x)
+V(x)\cdot \nabla U(x)\le -\frac{\gamma}2,\qquad\qquad\; x\in
\tilde{\Omega},
\]
i.e.,  $U$ is a uniform Lyapunov function in $\Omega$ with respect
to $\cal A=\{ A_\alpha\}$ when $|A_\alpha|\ll 1$.
\end{proof}
%{\color{red} As pointed out in the Introduction, we need to consider the subspace of $\tilde{\cal A}_b$ of $\tilde{\cal A}$
%equipped with the uniform topology in some situations. }

\subsection{Limit measures and stochastic stability} For a Borel set $\Omega\subset\R^n$, we denote by $M(\Omega)$ the
set of Borel probability measures on $\Omega$ furnished with the
{\em weak$^*$-topology}, i.e., $\mu_k\to \mu$ if and only if
\[
\int_{\Omega} f(x)\rmd \mu_k(x) \to \int_{\Omega} f(x) \rmd \mu(x),
\]
for every $f\in C_b(\Omega)$. It is well-known that $M(\Omega)$ with
the weak$^*$-topology is metrizable.

The following result is well-known (see e.g., \cite[Chapter II,
Theorem 6.1]{Parth}).
\begin{Proposition} \label{wt-kh} Let $\{\mu_\ell\}_{\ell=1}^\infty$ be a sequence in
$M(\Omega)$ and $\mu\in M(\Omega)$. Then the following statements
are equivalent.
\begin{enumerate}
\item[{\rm 1)}]  $\lim \limits_{\ell\rightarrow \infty}\mu_\ell=\mu$ under the
weak$^*$-topology.
\item[{\rm 2)}] $\limsup \limits_{\ell\rightarrow \infty}\mu_\ell(C)\le \mu(C)$ for
any closed subset $C$ of $\Omega$.
\item[{\rm 3)}] $\liminf \limits_{\ell\rightarrow \infty}\mu_\ell(W)\ge \mu(W)$ for
any open subset $W$ of $\Omega$.
 \item[{\rm 4)}]  $\lim \limits_{\ell\rightarrow \infty}\mu_\ell(B)= \mu(B)$ for
any Borel subset $B$ of $\Omega$ whose boundary has zero
$\mu$-measure.
\end{enumerate}
\end{Proposition}

\medskip

By Prokhorov's Theorem (see e.g. \cite[Theorem 5.1]{Bi}), a subset
$\mathcal{M}\subset M(\Omega)$ is relatively sequentially compact in
${\cal M}(\Omega)$ if it is {\em tight}, i.e., for any $\epsilon>0$
there exists a compact subset $K_\epsilon \subset \Omega$ such that
$\mu(\Omega \setminus K_\epsilon)<\epsilon$ for all $\mu\in
\mathcal{M}$. We note that if $\Omega$ is compact, then any subset
of $M(\Omega)$ is tight.
\medskip

\begin{Definition} {\em  Let  $\cal A=\{A_\alpha\}\subset \tilde{\cal
A}$ be an admissible  null family and $\Omega\subset\cal U$ be a
Borel set. Consider the set $\{\mu_{\alpha}\}$ of all stationary
measures corresponding to $\{{\cal L}_{A_\alpha}\}$. For each
$\mu_\alpha$, we denote
$\mu_\alpha^\Omega=\frac{\mu_\alpha|_{\Omega}}{\mu_\alpha(\Omega)}$
 when $\mu_\alpha(\Omega)>0$, called a {\em normalized stationary
measure  in $\Omega$}.
\begin{itemize}
\item[{\rm 1)}]   An $\cal A$-{\em limit
measure $\mu^\Omega$ in $\Omega$} is a sequential limit point, as
$A_\alpha\to 0$, of the set $\{\mu^\Omega_{\alpha}\}$ in
$M(\Omega)$.
 % in the
%set $M(\Omega)$ of Borel probability measures on $\Omega$ furnished
%with  the weak$^*$-topology.
%That is, there is a sequence $\{A_\ell
%\}_{\ell=1}^\infty\subset \cal A$ with $A_\ell\rightarrow 0$  as
%$\ell \rightarrow +\infty$ such that $\mu^\Omega_{A_\ell}
%\rightarrow \mu^\Omega$ in $M(\Omega)$ as $\ell \rightarrow
%+\infty$.
The set of all $\cal A$-limit measures in $\Omega$ is denoted by
${\cal M}_{\cal A}^\Omega$. If $\Omega=\cal U$, then we simply
denote ${\cal M}_{\cal A}^{\cal U}$  by ${\cal M}_{\cal A}$ and call
each measure $\mu\in {\cal M}_{\cal A}$ an {\em $\cal A$-limit
measure}.

\item[{\rm 2)}] The set $\{\mu_\alpha^\Omega\}$ is said to be
 {\em $\cal A$-sequentially null compact in $M(\Omega)$}
 if  any sequence $\{\mu_l^\Omega\}$ in the set with $A_l\to 0$  is relatively compact in $M(\Omega)$.
%if  $\{\mu^\Omega_{A_\ell}\}_{\ell=1}^\infty$ is relatively
%sequentially compact in $M(\Omega)$ for any sequence
%$\{A_\ell\}_{\ell=1}^\infty\subset \mathcal{A}$ with
%$A_\ell\rightarrow 0$.
\end{itemize}

%\item[{\rm 3)}] Let $\{\cal U,\cal A\}$ be admissible and denote $\{\mu_\alpha\}$ as the set
%of all stationary measures corresponding to $\{{\cal L}_{A_\alpha}\}$ in
%$\cal U$.
%\begin{itemize}
%\item[{\rm i)}] An $\cal A$-{\em limit
%measure $\mu$ in $M(\mathcal{U})$} is a sequential limit point, as
%$A_\alpha\to 0$, of the set $\{\mu_\alpha\}$ of
 %stationary  measures.
 % $\{\mu_\alpha\}$  in the
%set $M(\mathcal{U})$ of Borel probability measures furnished with  the
%weak$^*$-topology. The set of all $\cal A$-limit measures in $M(\mathcal{U})$
%is denoted by ${\cal M}_{\cal A}$.

%\item[{\rm ii)}] The set $\{\mu_\alpha\}$ is called {\em $\cal A$-sequentially null compact in $M(\mathcal{U})$},
%if  $\{\mu_{A_\ell}\}_{\ell=1}^\infty$ is relatively sequentially
%compact in $M(\cal U)$ for any  sequence
%$\{A_\ell\}_{\ell=1}^\infty\subset \mathcal{A}$ with
%$A_\ell\rightarrow 0$.
%\end{itemize}

}
\end{Definition}

\medskip

\begin{Remark}\label{tt} {\em Let  $\cal A=\{A_\alpha\}\subset \tilde{\cal A}$ be an admissible  null family on
$\cal U$ and $\Omega\subset \cal U$ be a connected open set.

 1) By
the aforementioned regularity theorem in \cite{BKR},
$\mu_\alpha(\Omega)>0$ in this situation.

2)  We note that when $\Omega=\cal U$, the normalized stationary
measure $\mu_\alpha^{\cal U}$ coincides with the original stationary
measure $\mu_\alpha$ corresponding to ${\cal L}_{A_\alpha}$, for
each $\alpha$.

3) If   the set $\{\mu_\alpha^\Omega\}$ is $\cal A$-sequentially
null compact in $M(\Omega)$,  then ${\cal M}^\Omega_{\cal A}$ is
clearly non-empty. In general,  ${\cal M}^\Omega_{\cal A}$ can be an
empty set. However, when $\Omega$ is a pre-compact open subset of
$\cal U$ (i.e., its closure in $\mathbb{R}^n$ is a compact subset of
$\cal U$), it follows from the regularity theorem again that each
stationary measure corresponding to ${\cal L}_{A_\alpha}$  is
regular, hence each $\mu_\alpha^\Omega$ is also a probability
measure on $\bar\Omega$. This enables us to consider the set ${\cal
M}_{\cal A}^{\bar\Omega}$ of all $\cal A$-limit measures of
$\{\mu_\alpha^\Omega\}$ in $M(\bar\Omega)$, which is always
non-empty because $\{\mu_\alpha^\Omega\}$ is tight, hence relatively
sequentially compact in $M(\bar\Omega)$.

 4) If a sub-domain $\Omega\subset\cal U$ is considered in
\eqref{j2} instead of $\cal U$, then one can speak of a stationary
measure corresponding to ${\cal L}_{A}$ in $\Omega$.  A normalized
stationary measure $\mu_\alpha^\Omega$ in $\Omega$ is necessarily a
stationary measure corresponding to $\{{\cal L}_{A_\alpha}\}$ in
$\Omega$. When $\Omega$ is a pre-compact open subset of $\cal U$,
stationary measures corresponding to $\{{\cal L}_{A_\alpha}\}$ in
$\Omega$ other than the normalized ones $\{\mu_\alpha^\Omega\}$ also
exist. In fact, depending on the boundary conditions imposed, there
are infinitely many such stationary measures in $\Omega$ (see
\cite{BR01,JM1} for details). Though most of our results in this
paper hold for all stationary measures corresponding to $\{{\cal
L}_{A_\alpha}\}$ in $\Omega$, we only work with the normalized ones
$\{\mu_\alpha^\Omega\}$ because they are noise or diffusion
relevant.}
\end{Remark}

\medskip

We now define the following notion of stochastic stability with
respect to a given admissible null family of noise perturbations.
\medskip

\begin{Definition}~\label{2.5} {\em Let $\Omega\subset \cal U$ be  a connected open
set and $\cal A=\{A_\alpha\}\subset \tilde{\cal A}$ be an
invariant and admissible null family. Let
$\{\mu_{\alpha}\}$ be the set of all stationary measures
corresponding to $\{{\cal L}_{A_\alpha}\}$ and
$\{\mu_\alpha^\Omega\}$ be the set of normalized stationary measures
 on $\Omega$.

1) A compact invariant set  $\cal J\subset \Omega$  of $\varphi^t$
is said to be {\em relatively $\cal A$-stable  in $\Omega$}
 if for any $\epsilon>0$ and any open neighborhood $W$ of $\cal J$ in $\mathcal{U}$
  there exists a $\delta>0$ such that $\mu_\alpha^\Omega(\Omega\setminus
 W)<\epsilon$  whenever $d(A_\alpha,0) < \delta$. An invariant measure $\mu$ of $\varphi^t$ in $\Omega$ is said to
 be {\em relatively $\cal A$-stable  in $\Omega$}  if
$\{\mu_\alpha^\Omega\}$  converges to $\mu$ in $M(\Omega)$  as
$A_\alpha \to 0$.

  2) A compact invariant
 set or invariant measure is said to be {\em relatively $\cal
 A$-unstable in $\Omega$} if it is not relatively $\cal A$-stable in
 $\Omega$.

 3)  A relatively $\cal A$-stable  (resp. relatively $\cal A$-unstable) compact invariant
 set or invariant measure in $\cal U$ is simply said to be {\em $\cal
 A$-stable} (resp. {\em $\cal A$-unstable}).  A compact invariant set $\cal J$ is said to be {\em strongly $\cal
 A$-unstable} if ${\rm supp} (\mu)\cap \cal J=\emptyset$ for all $\mu\in {\cal M}_{\cal
 A}$.
 }
\end{Definition}
\medskip

\begin{Remark} {\em
 1) Clearly the admissibility of a null family $\cal A$ is necessary in defining (relative) $\cal A$-stability,
otherwise the (relative) $\cal A$-stability is void.

 2) The consideration of stationary measures instead of invariant
measures allows a general characterization of stochastic stability
because they often exist regardless whether \eqref{sde} generates a
diffusion process or \eqref{fp} generates a generalized diffusion
process, not talking about the fact that even when a diffusion or a
generalized diffusion process can be defined its invariant measures
need not exist. Such a consideration is natural in the sense that
compact invariant sets, respectively invariant measures,  of
\eqref{ode} are indeed stationary with respect to its induced flow
on compact sets, respectively on the space of Borel probability
measures.

3) The invariance of $\cal A$ is also necessary in defining
(relative) $\cal A$-stability because for such a stability one would
like to only restrict the consideration to  ``noise relevant"
stationary measures. This is particularly so when $\cal U$ is
bounded. In this case, stationary measures corresponding to ${\cal
L}_{A}$ for any $A\in \tilde{\cal A}$ always exist, but depending on
the boundary conditions imposed, many of which need not be
associated with true stationary processes if $\cal A$ fails to be
invariant (see also Remark~\ref{well-post} 2) and Remark~\ref{tt}
4)).

%d) Let $\cal A=\{A_\alpha\}$ be an {\color{red}invariant null} family.
%Then as $t\to+\infty$, the set of limit points
% of orbits of the measure-theoretic counter part of \eqref{fp}, if exists,
% consists of stationary measures corresponding to $\{{\cal
%L}_{A_\alpha}\}$. However, we do not know whether it contains all
%stationary measures corresponding to $\{{\cal L}_{A_\alpha}\}$. If
%it does not, then {\color{red}this set} {\color{blue}(change ``this set" to `` the set $\{\mu_\alpha\}$")}
% give more ``noise relevant"
%stationary measures which can be used in the definition of $\cal
%A$-stability. This is certainly an interesting problem {\color{red}worthy of} a further exploration.

%{\color{blue}(suggest change to `` This is certainly an interesting problem {\bf worthy of} a further exploration.
%If it does not, then {\bf the set $\{\mu_\alpha\}$} would give more ``noise relevant"
%stationary measures which can be used in the definition of $\cal
%A$-stability.")}

%4) Relative $\cal A$-stability is a local property which is
% only meaningful when considering  a proper subset $\Omega$ of $\cal U$. This
%is  because when $\Omega=\cal U$, ${\rm supp} (\mu)\cap \cal
%U\ne\emptyset$ is automatic for any $\mu\in {\cal M}_{\cal A}$, and
%consequently the notions of relative $\cal A$-stability and $\cal
%A$-stability in $\cal U$ coincide. }

4) Relative $\cal A$-stability of a compact invariant set or an
invariant measure of $\varphi^t$ in $\Omega$ only says that the
normalized  stationary measures $\{\mu_\alpha^\Omega\}$  rather than
the restricted stationary measures $\{\mu_\alpha|_\Omega\}$ are
``attracted" to the set or the measure as $A_\alpha\to 0$,  i.e., it
can happen that eventually all or part of the original stationary
measures $\{\mu_\alpha\}$ corresponding to $\{{\cal L}_{A_\alpha}\}$
 still ``escape" from $\Omega$ and concentrate elsewhere.
 To ensure the eventual concentration of all restricted stationary
measures $\{\mu_\alpha|_\Omega\}$ in $\Omega$, an additional
condition ${\rm supp} (\mu)\cap \Omega\ne\emptyset$, $\mu\in {\cal
M}_{\cal
 A}$ need to be imposed.

 %{\color{red} i.e.}, if $\cal J$ is an $\cal A$-stable
%compact invariant set of $\varphi^t$ in $\Omega$, then, as
%$A_\alpha\to 0$, all $\{\mu_\alpha|_\Omega\}$ will be eventually
%concentrated on $\cal J$, and if $\mu$ is an $\cal A$-stable
%invariant measure of $\varphi^t$ in $\Omega$, then, as $A_\alpha\to
%0$, $\{\mu_\alpha|_\Omega\}$ will have limit, say $\mu_*$, and
%$\mu=\frac{\mu_*}{\mu_*(\Omega)}$.

% 2) {\color{red} A compact invariant set or an invariant measure in
% $\Omega$}
 %is said  to be {\em $\cal A$-stable in $\Omega$}
 %if {\color{red} it is relatively $\cal A$-stable in $\Omega$ and} . {\color{red} A compact invariant set or an invariant measure in
 %$\Omega$ is said to be {\em $\cal A$-unstable in $\Omega$} if it is
% not $\cal A$-stable in $\Omega$.}
 }

\end{Remark}

\subsection{Harnack inequality}

Consider the differential operator
\begin{equation}\label{h1}
Lu:=\partial_i (a^{ij}(x)\partial_j u+b^i(x)u)+ c^i(x)\partial_i
u+d(x)u,
\end{equation}
where $a^{ij}$, $b^i$, $c^i$, $d$, $i,j=1,\cdots,n$, are measurable
and bounded functions on a  connected open set $\Omega\subset\R^n$.
Assume that
\begin{equation}\label{h0}
 a^{ij}(x)\xi_i\xi_j\ge \lambda |\xi|^2,\qquad x\in \Omega,
\; \xi\in\R^n
\end{equation}
for some constant $\lambda >0$, and   there exist  constants
$\Lambda>0$, $\nu\ge0$ such that
\begin{equation}\label{h2}
 \sum_{i,j}|a^{ij}(x)|^2\le \Lambda^2, \quad
\lambda^{-2}\sum_{i}(|b^i(x)|^2+|c^{i}(x)|^2)+\lambda^{-1}|d(x)|\le
\nu^2
\end{equation}
for all $x\in\Omega$.

The following Harnack inequality  is well-known (see Theorem 8.20 of
\cite{GT}).

\begin{Proposition} \label{Harn} {\rm (Harnack inequality)}
Assume that the operator $L$ satisfies  \eqref{h0} and \eqref{h2}.
Let $u\in W^{1,2}(\Omega)$ be any non-negative solution of
\[
Lu(x)=0,\qquad x\in\Omega. \]  Then for any ball $B_{4R}(y)\subset
\Omega$, we have
\[
\sup_{B_R(y)}u\le C \inf_{B_R(y)}u,
\]
where the constant $C$ can be estimated by
\[
C \le C_0^{(\Lambda/\lambda+\nu R)}
\]
for some constant $C_0=C_0(n)$ depending only on $n$.
\end{Proposition}

\section{General properties  of limit measures}

In this section, we consider some general limit properties of
stationary measures  of a family of Fokker-Planck equations as
diffusion coefficients tend to zero. Among these properties, the
concentration of limit measures  will play an important role in
characterizing stochastic stability of compact invariant  sets of
the corresponding unperturbed deterministic system at both local and
global levels.

For a fixed vector field $V\in  C(\cal U,\R^n)$ in \eqref{ode}, we
consider the  class $\tilde{\cal A}$ of diffusion matrices defined
in \eqref{admissible}. When $V\in C^1(\cal U,\R^n)$, \eqref{ode}
generates a $C^1$ local flow on $\cal U$ which will be denoted by
$\varphi^t$ through this and the next section.

This section (as well as next section) uses some basic dynamics
notions of a local flow  such as dissipation and anti-dissipation,
(strong) attractors and repellers, (entire, weak) Lyapunov and
anti-Lyapunov functions, and (positive, negative) invariant
measures. The definition of these notions and some of their
fundamental properties are reviewed in the Appendix (Section 6).

\medskip

\subsection{Invariance of limit measures} We first  give some new characterizations of
invariant and semi-invariant measures of $\varphi^t$ in a subset
which are useful in the study of limit behaviors of stationary
measures of Fokker-Planck equations.

\medskip

\begin{Proposition}\label{simple-1}
Assume  $V\in C^1(\cal U,\R^n)$ and let $\Omega\subset \cal U$ be an
open invariant
 set of $\varphi^t$. Then  $\mu\in M(\Omega) $ is an invariant measure  of
$\varphi^t$ on $\Omega$  if and only if
\begin{equation}\label{T1}
\int_{\Omega} V(x) \cdot \nabla h(x)~\rmd \mu( x)=0,\qquad h\in
C_0^1(\Omega).
\end{equation}
\end{Proposition}

\begin{proof}
It follows from Riesz Representation Theorem that $\nu_1=\nu_2\in
M(\Omega)$  if and only if
\begin{equation}\label{f=eq}
\int_{\Omega} h(x) ~{\rm d}\nu_1( x)=\int_{\Omega}h(x) ~{\rmd}\nu_2(
x),\qquad\; h \in C_b(\Omega).
\end{equation}
As any $\nu\in M(\Omega)$ is a Borel regular measure,  \eqref{f=eq}
is equivalent to
\begin{equation}\label{f=eqa}
\int_{\Omega} h(x) ~{\rm d}\nu_1( x)=\int_{\Omega}h(x) ~{\rmd}\nu_2(
x),\qquad\; h \in C_0(\Omega).
\end{equation}

  For a given $t\in\R$, define $\mu^t\in
M(\Omega)$: $\mu^t(B)=:\mu(\varphi^{-t}(B))$ for any Borel set $B
\subset \Omega$. Then the invariance of $\mu$ is equivalent to
$\mu=\mu^t$ for any $t\in \mathbb{R}$, which, by \eqref{f=eq} and
\eqref{f=eqa}, is equivalent to
\begin{equation}\label{eq-csup}
\int_{\Omega}h(x)~ {\rmd}\mu( x)=\int_{\Omega}h(x) ~ {\rmd}\mu^t(
x)=\int_{\Omega}h(\varphi^t(x)) ~{\rmd}\mu( x), \;\, t\in \R,
\end{equation}
for all $h\in C_b(\Omega)$ or all $h\in C_0(\Omega)$.

For any $h\in C_0^1(\Omega)$, consider the function
\[
f_h(t):=\int_{\Omega} h(\varphi^t(x)) ~{\rmd} \mu(x),\qquad t\in \R.
\]
Then, for any $t, s\in \R$, we have by the flow property that
\[
f_h(t+s)=\int_{\Omega} h(\varphi^{(t+s)}(x)) ~{\rmd} \mu(x)=
\int_{\Omega} h\circ \varphi^t(\varphi^s(x)) ~{\rmd}
\mu(x)=f_{h\circ\varphi^t}(s).
\]
It follows that
\begin{equation}\label{T0}
f_h'(t)=f_{h\circ \varphi^t}'(0),\qquad h\in C_0^1(\Omega),\;\, t\in
\R.
\end{equation}

%\begin{eqnarray}
%f_h'(t)&=&\int_\Omega V(\varphi^t(x))\cdot \nabla
%h(\varphi^t(x))\rmd \mu(x)=\int_\Omega V(\varphi^{(t+s)}(x))\cdot
%\nabla h(\varphi^{(t+s)}(x))\big |_{s=0}\rmd \mu(x)
%\nonumber\\
%&=&\int_{\varphi^t(\Omega)} V(\varphi^s(x))\cdot \nabla
%h(\varphi^s(x))\big |_{s=0}\rmd \mu^t(x)\nonumber\\
%&=&\int_{\varphi^t(\Omega)} V(\varphi^s(x))\cdot \nabla (h\circ
%\varphi^t)(\varphi^s(x))\big |_{s=0}\rmd \mu(x)=f_{h\circ
%\varphi^t}'(0),\qquad t\in \R.\label{T0}
%\end{eqnarray}

If $\mu$ is an invariant measure of $\varphi^t$ in $\Omega$, then
for any $h\in C_0^1(\Omega)$,  \eqref{eq-csup} implies that
$f_h(t)=f_h(0)$ for all $t\in \R$. Taking derivatives yields that
$f_h'(t)\equiv 0$. In particular, $f_h'(0)= 0$, i.e., \eqref{T1}
holds.

Conversely, suppose that \eqref{T1} holds.  For any
$h\in C_0^1(\Omega)$ and any $t\in \R$, using the fact that
$\varphi^t:\Omega\to \Omega$ is a $C^1$ diffeomorphism, it is easy
to see that $h\circ \varphi^t\in C_0^1(\Omega)$. Applications of
\eqref{T0} and \eqref{T1}  with $h\circ \varphi^t$ in place $h$
yield that
\[
f_h'(t)= f_{h\circ\varphi^t}'(0)= \int_{\Omega}V(x) \cdot \nabla
(h\circ\varphi^t)(x)~ \rmd\mu( x) =0,
\]
i.e.,
\[
\int_{\Omega} h(x) ~\rmd\mu( x)=\int_{\Omega}h(\varphi^t(x))
~\rmd\mu( x), \qquad\; h\in C_0^1(\Omega),\;\, t\in \R.
\]
Since $C_0^1(\Omega)$ is dense in $C_0(\Omega)$, \eqref{eq-csup}
holds. Hence $\mu$ is an invariant measure of $\varphi^t$ in
$\Omega$.
\end{proof}
\medskip

\begin{Proposition}\label{simple-1-1}
Assume  $V\in C^1(\cal U,\R^n)$ and let $\Omega\subset \cal U$ be an
 open, pre-compact, positively {\rm (}resp. negatively{\rm )}
invariant
 set of $\varphi^t$.   Then  $\mu\in M(\Omega) $ is a positively {\rm (}resp. negatively{\rm )} invariant measure  of
$\varphi^t$ on $\Omega$  if and only if
\[
\int_{\Omega} V(x) \cdot \nabla h(x)~\rmd \mu( x)=0,\qquad h\in
C_b^1(\Omega).
\]
\end{Proposition}

\begin{proof}
 The proof follows from the same argument as that of
Proposition~\ref{simple-1} with $C_b(\Omega)$, $C_b^1(\Omega)$,
$\R_+$ (resp. $\R_{-}$) in places of $C_0(\Omega)$, $C_0^1(\Omega)$,
$\R$ respectively.
\end{proof}

Part a) of  Theorem~A  follows from part b) of the following result
when taking $\Omega=\cal U$.

\begin{Theorem}\label{tight0-1}   Let $\cal
A=\{A_\alpha\}\subset \tilde{\cal A}$ be an admissible null family
and $\Omega\subset \cal U$ be a connected open set. Then the
following holds for any $\mu\in \cal M_{\cal A}^\Omega$.
\begin{itemize}
\item[{\rm a)}] $\int_{\Omega} V(x) \cdot \nabla h(x)~\rmd \mu( x)=0$ for all $h\in
C_0^1(\Omega)$.

\item[{\rm b)}] Suppose that $V\in C^1(\cal U,\R^n)$ and that $\Omega$ is an invariant set of
$\varphi^t$. Then $\mu$ is an invariant measure of $\varphi^t$ on
$\Omega$.
\item[{\rm c)}] Suppose that $V\in C^1(\cal U,\R^n)$, and that there exists an open,  pre-compact, positively {\rm (}resp.
negatively{\rm)}
 invariant set  of
$\varphi^t$ in $\Omega$ containing ${\rm supp}(\mu)$. Then $\mu$ is
an invariant measure of $\varphi^t$ on $\Omega$.
\end{itemize}
\end{Theorem}
\begin{proof}  Denote $\{\mu_\alpha^\Omega\}$ as the set of all normalized stationary measures corresponding to
$\{{\cal L}_{A_\alpha}\}$ in $\Omega$. Let
 $\{A_l=(a^{ij}_l)\}\subset \cal A$ and
$\{\mu_l=:\mu_l^\Omega\}\subset \{\mu_\alpha^\Omega\}$ be sequences
such that $A_l\to 0$ and $\mu_l\to \mu$ in $M(\Omega)$, as
$l\to\infty$.

 We note by the definition of stationary measures that
\begin{equation}\label{limitM-11}
\di_{\Omega} (a^{ij}_l\partial^2_{ij} h(x) +V^i\partial_i h(x))~\rmd
\mu_l(x) =0,\quad\quad  h\in C_0^\infty(\Omega),
\end{equation}
for all $l$.  Using the  uniform convergence of $(a^{ij}_l)$ to $0$
on any compact subsets of $\cal U$ and the weak$^*$-convergence of
$\mu_l$ to $\mu$, we have by simply taking limit $l\to\infty$ in
\eqref{limitM-11} that
\[
\int_{\Omega} V(x)\cdot \nabla h(x)~ \rmd\mu( x)=0,\qquad\; h\in
C_0^\infty(\Omega).
\]
It follows that a) holds because $C_0^\infty(\Omega)$ is dense in
$C_0^1(\Omega)$.

When $\Omega$ is an invariant set of $\varphi^t$, the invariance of
$\mu$ follows from Proposition~\ref{simple-1}. This proves b).

To prove c), we let $\tilde\Omega$ be an open,  pre-compact,
positively {\rm (}resp. negatively{\rm)}
 invariant set  of
$\varphi^t$ in $\Omega$ containing ${\rm supp}(\mu)$. For any
$\tilde h\in C^1_b(\tilde\Omega)$, we let $h\in C_0^1(\Omega)$ be an
extension of $\tilde h$. Then by a), we have
\[
\int_{\tilde\Omega} V(x)\cdot \nabla \tilde h(x)~ \rmd\mu(
x)=\int_{\Omega} V(x)\cdot \nabla h(x)~ \rmd\mu( x)=0.
\]
It follows from Proposition~\ref{simple-1-1} that $\mu$ is a
positively (resp. negatively) invariant measure of $\varphi^t$ in
$\tilde\Omega$.   It follows from
Proposition~\ref{invariant-measure} that $\mu$ is an invariant
measure of $\varphi^t$ in $\tilde\Omega$ supported on the
$\omega$-limit set $\omega(\tilde\Omega)$ (resp. $\alpha$-limit set
$\alpha(\tilde\Omega)$) of $\tilde\Omega$. It follows that $\mu$ is
an invariant measure of $\varphi^t$ on $\Omega$ because
$\mu(\Omega\setminus \tilde\Omega)=0$.
\end{proof}

\subsection{Stochastic counterpart of the LaSalle invariance
principle} For a deterministic dynamical system, the classical
LaSalle invariance principle  (see Proposition~\ref{LaSalle}) is an
important tool in locating $\omega$-limit sets using a Lyapunov-like
function. For the stochastic system \eqref{sde}, we show below that
a similar principle can be obtained for locating the support of a
limit measure.

For a Borel  set $\Omega\subset \cal U$ and an admissible null
family $A=\{A_\alpha\}\subset \tilde{\cal A}$, we denote
\begin{equation}\label{S1}
{\cal J}_{\cal A}^\Omega :=\overline{ \cup \{{\rm supp}(\mu): \mu
\in {\cal M}_{\cal A}^\Omega\}}.
\end{equation}
In the case $\Omega=\cal U$, we simply denote ${\cal J}_{\cal
A}^{\cal U}$ by ${\cal J}_{\cal A}$.

% Indeed, the following result can be
%viewed as a stochastic counterpart of the LaSalle invariance
%principle.
\medskip

\begin{Proposition}\label{S-LaSalle} Let $\cal A=\{A_\alpha\}\subset
\tilde{\cal A}$ be an admissible null family and $\Omega\subset\cal
U$ be a connected open set. Then the following holds.
\begin{itemize}
\item[{\rm a)}] If \eqref{ode} admits
an entire weak Lyapunov {\rm (}resp. anti-Lyapunov{\rm )} function
$U$ in $\Omega$,  then any $\mu\in {\cal M}_{\cal A}^\Omega$ with
compact support satisfies
\[ {\rm supp} (\mu)\subset  S=\{x\in
{\Omega}: V(x)\cdot \nabla U(x)=0\}.
\]
\item[{\rm b)}] If the  set ${\cal J}_{\cal A}^\Omega$ is contained in a compact set $\cal J\subset \Omega$
and  there is an entire weak Lyapunov {\rm (}resp. anti-Lyapunov{\rm
)} function $U_0$ of \eqref{ode} defined in a neighborhood of ${\cal
J}$, then any $\mu\in {\cal M}_{\cal A}^\Omega$ satisfies
\[ {\rm supp} (\mu)\subset  S_0=\{x\in
{\cal J}: V(x)\cdot \nabla U_0(x)=0\}.
\]
\end{itemize}
\end{Proposition}

\begin{proof} To prove a), we let $\rho_0>0$ be such that the  $\rho_0$-sublevel set
$\Omega_{\rho_0}$ of $U$ contains the support of $\mu$. Applying
Theorem~\ref{tight0-1} a) to a function $h\in C^1_0(\Omega)$
satisfying $h\equiv U$ on $\Omega_{\rho_0}$ yields that
\[
0=\int_{\Omega} V(x)\cdot \nabla h(x)~  \rmd\mu(
x)=\int_{\Omega_{\rho_0}} V(x)\cdot \nabla U(x)~\rmd \mu(
x)=\int_{\Omega_{\rho_0}\cap (\Omega\setminus S)} V(x)\cdot \nabla
U(x)~ \rmd\mu( x).
\]
Since $V(x)\cdot \nabla U(x)$ is  of constant sign on $
\Omega_{\rho_0}\cap (\Omega\setminus S)$, $\mu(\Omega_{\rho_0}\cap
(\Omega\setminus S))=0$, i.e., $ {\rm supp}(\mu) \subset
\Omega_{\rho_0}\cap S\subset S$.

To prove b), we note by a) that ${\rm supp}(\mu) \subset S_0=:\{x\in
{\Omega_0}: V(x)\cdot \nabla U_0(x)=0\}$, where $\Omega_0$ is a
neighborhood of $\cal J$ in which $U_0$ is defined. By the
definition of $\mathcal J_{\mathcal A}^\Omega$ in \eqref{S1}, we
also have ${\rm supp}(\mu) \subset \mathcal J$. It follows that
${\rm supp}(\mu) \subset \mathcal J \cap S_0$, from which b)
follows.
\end{proof}

\medskip

 We note that part a) of Proposition~\ref{S-LaSalle} with
$\Omega=\cal U$ is precisely  the part b) of
 Theorem~A.

\subsection{Some general properties on $\cal A$-stability}

 We first give an equivalent condition for a compact
invariant set of $\varphi^t$ to be relatively $\cal A$-stable or
$\cal A$-stable.
\medskip

\begin{Proposition}\label{stable} Let $\Omega\subset \cal U$ be  a connected open set,
$\cal J\subset \Omega$ {\rm (}resp. $\cal J\subset \cal U${\rm )} be
a compact invariant set  of $\varphi^t$,  and $\cal
A=\{A_\alpha\}\subset \tilde{\cal A}$ be an invariant
and admissible null family. Then $\cal J$ is relatively $\cal
A$-stable in $\Omega$ {\rm (}resp. $\cal A$-stable{\rm)} if and only
if the set $\{\mu_\alpha^\Omega\}$ of all normalized {\rm (}resp.
the set $\{\mu_\alpha\}$ of all {\rm
  )} stationary measures corresponding to
  $\{{\cal L}_{A_\alpha}\}$  in $\Omega$ {\rm(}resp. in $\cal U${\rm)}
  is $\cal A$-sequentially null compact in $M(\Omega)$ {\rm (}resp.
   in $M(\cal U)${\rm )} and  ${\cal J}_{\cal A}^\Omega\subset \cal J$ {\rm (}resp.
${\cal J}_{\cal A}\subset \cal J${\rm)}, i.e., ${\rm
supp}(\mu)\subset \cal J$ for all $\mu\in {\cal M}^\Omega_{\cal A}$
{\rm (}resp. $\mu\in {\cal M}_{\cal A}${\rm )}.
\end{Proposition}

\begin{proof} To show the sufficiency, we suppose for contradiction
that $\cal J$ is not relatively  $\cal A$-stable in $\Omega$. Then
there exists an $\epsilon_0>0$, an open neighborhood $W$ of
$\mathcal{J}$ in $\Omega$ and a sequence $\{A_l\}\subset \cal A$
such that $A_l\to 0$ but $\mu^\Omega_{A_l}(\Omega\setminus W)\ge
\epsilon_0$ for all $l$. Since $\{\mu^\Omega_{A_l}\}$ is relatively
sequentially compact, we may assume without loss of generality that
$\mu^\Omega_{A_l}$ converges, say, to some $\mu\in {\cal
M}^\Omega_{\cal A}$. Since $\Omega\setminus W$ is a closed subset of
$\Omega$ under the restricted topology, Proposition \ref{wt-kh}
implies that $\mu(\Omega\setminus W)\ge \limsup_{l\rightarrow
\infty} \mu^\Omega_{A_l}(\Omega\setminus W)\ge \epsilon_0$. Thus,
$\text{supp}(\mu) \nsubseteq \cal J$, a contradiction to the fact
that ${\cal J}_{\cal A}^\Omega\subset \cal J$.

To show the necessity, we let $\{A_l\}\subset \cal A$ be a sequence such that $A_l\to
 0$.  Let $\epsilon>0$ be
given. We also take any  compact neighborhood $W$ of $\cal J$ in
$\Omega$. Since $\cal J$ is relatively $\cal A$-stable, there exists
a positive integer $N$ such that
\begin{equation} \label{T2}
\mu_{A_\ell}^\Omega(\Omega\setminus W)<\epsilon, \qquad\quad \ell\ge
N,
\end{equation}
 i.e., $\mu_{A_\ell}^\Omega(W)\ge 1-\epsilon$ for
all $\ell\ge N$. Since each $\mu_{A_\ell}^\Omega$ is a Borel regular
measure, there is a compact subset $B$ of $\Omega$ such that
$\mu_{A_\ell}^\Omega(B)\ge 1-\epsilon$ for all $1\le \ell \le N$.
 Let $K=B\cup W$. Then $K$ is a compact subset of $\Omega$ and  $\mu_{A_\ell}^\Omega(K)\ge 1-\epsilon$ for all $\ell$.
 This shows that $\{\mu_{A_\ell}^\Omega\}_{\ell=1}^\infty$ is tight, hence relatively  compact in
$M(\Omega)$.

Let $\mu\in {\cal M}^\Omega_{\cal A}$. Then there exists a sequence
$\{A_l\}\subset \cal A$ with $A_l\to 0$ such that
$\lim_{\ell\rightarrow \infty}\mu_{A_\ell}^\Omega=\mu$ under the
weak$^*$-topology. Again, using the relative $\cal A$-stability of
$\cal J$ in $\Omega$, we see that for any $\epsilon>0$ and any
compact neighborhood $W$ of $\mathcal{J}$ in $\Omega$, \eqref{T2}
holds with the present $\epsilon, W$, and $\{\mu_{A_\ell}^\Omega\}$.
  Since
$\Omega\setminus
 W$ is an open subset of $\Omega$, Proposition \ref{wt-kh} implies that
 $$\mu(\Omega\setminus
 W)\le \liminf_{\ell\rightarrow \infty}\mu_{A_\ell}^\Omega(\Omega\setminus
 W)\le \epsilon.$$
 Since $\epsilon$ and $W$ are arbitrary,  we have $\mu(\Omega\setminus \mathcal{J})=0$ by the continuity of probability measures,
 i.e., ${\rm supp}(\mu)\subset \cal J$.
\end{proof}

\begin{Corollary}\label{min-A-S} Assume $V\in C^1(\cal U,\R^n)$. Let  $\cal
A=\{A_\alpha\}\subset \tilde{\cal A}$ be an invariant
and admissible null family and $\Omega\subset \cal U$ be a
connected open set. If ${\cal J}_{\cal A}^\Omega$ {\rm (}resp.
${\cal J}_{\cal A}${\rm )} is a relatively $\cal A$-stable set in
$\Omega$ {\rm (}resp. $\cal A$-stable set{\rm )}, then it is the
smallest relatively $\cal A$-stable set in $\Omega$ {\rm (}resp.
$\cal A$-stable set{\rm )}, i.e., it contains no proper relatively
$\cal A$-stable subset in $\Omega$ {\rm (}resp. $\cal A$-stable
set{\rm )}.
\end{Corollary}

\begin{proof}
It follows from Proposition~\ref{stable} immediately.
\end{proof}

The following result establishes some connections between
(relatively) $\cal A$-stable sets and measures.
\medskip

\begin{Proposition}~\label{supp-mu} Let $\cal
A=\{A_\alpha\}\subset \tilde{\cal A}$ be an invariant
and admissible null family and $\Omega\subset \cal U$ be  a
connected open set. Then the following holds.
\begin{itemize}
\item[{\rm a)}] Let $\mu\in M(\Omega)$ {\rm (}resp. $\mu\in M(\cal
U)${\rm )} be an invariant measure of $\varphi^t$ with compact
support. If $\mu$ is relatively $\cal A$-stable in $\Omega$ {\rm
(}resp. $\cal A$-stable{\rm)}, then so is ${\rm supp} (\mu)$ as a
compact invariant set.
\item[{\rm b)}] Assume $V\in C^1(\cal U,\R^n)$. Let $\cal J\subset
\Omega$ {\rm(}resp. $\cal J\subset \cal U${\rm )} be a compact
invariant set of $\varphi^t$ which is uniquely ergodic and contained
in an open, pre-compact, positively or negatively invariant subset
of $\varphi^t$ in $\Omega$ {\rm (}resp. in $\cal U${\rm )}. If $\cal
J$ is relatively $\cal A$-stable in $\Omega$ {\rm (}resp. $\cal
A$-stable{\rm)}, then so is the unique ergodic measure on $\cal J$.
\end{itemize}
\end{Proposition}
\begin{proof}   Denote $\{\mu_\alpha^\Omega\}$ as the set of all
normalized  stationary measures in $\Omega$ corresponding to the
family $\{{\cal L}_{A_\alpha}\}$.

a) We note   that
 ${\rm supp} (\mu)$ is a compact invariant set  of
$\varphi^t$ in $\Omega$. Since $\mu$ is relatively $\cal A$-stable
in $\Omega$, $\{\mu_\alpha^\Omega\}$  converges to $\mu$ in
$M(\Omega)$ as $A_\alpha \to 0$. Let $W$ be an open neighborhood of
$\text{supp}(\mu)$ in $\Omega$. Then it follows from Proposition
\ref{wt-kh} and the fact $\mu(\Omega\setminus \text{supp}(\mu))=0$
that
$$\limsup_{A_\alpha \to 0}\mu_\alpha^{\Omega}(\Omega\setminus W)\le \mu(\Omega\setminus W)=0.$$
 Thus for any $\epsilon>0$, there exists $\delta>0$ such that
$\mu_\alpha^\Omega(\Omega\setminus
 W)<\epsilon$  whenever $d(A_\alpha,0)<\delta$, i.e., ${\rm supp}
 (\mu)$ is relatively $\cal A$-stable in $\Omega$.

b) Let $\nu$ be the unique invariant measure of $\varphi^t$ on
$\mathcal{J}$. Suppose for contradiction that  $\nu$ is not
relatively  ${\cal A}$-stable in $\Omega$. Then there are sequences
 $\{A_l=(a^{ij}_l)\}_{\ell=1}^\infty\subset \cal A$ and
$\{\mu_l^\Omega\}\subset \{\mu_\alpha^\Omega\}$ such that $A_l\to 0$
but $\mu_l^\Omega$ is not convergent to $\nu$ in $M(\Omega)$, as
$l\to\infty$. Since $\cal J$ is relatively $\cal A$-stable in
$\Omega$, Proposition~\ref{stable} implies that
$\{\mu_{A_\ell}^\Omega\}$ is relatively sequentially compact.
Without loss of generality, we assume  $\lim_{\ell\to\infty}
\mu_{A_\ell}^\Omega=\nu^*\ne \nu$. By Theorem~\ref{tight0-1} c) and
Proposition~\ref{stable},   $\nu^*$ is an invariant measure of
$\varphi^t$ on $\cal J$. But since $\cal J$ is uniquely ergodic, we
must have $\nu^*=\nu$, a contradiction.
\end{proof}

%{\color{blue} (As a definition of ${\cal J}_{\cal A}^{\bar\Omega}$
%was added, a) is okey now)}

We now consider the case when there exists a uniform Lyapunov
function corresponding to $\{{\cal L}_{A_\alpha}\}$
 in $\Omega$. In this case, we recall from
Proposition~\ref{UNIF} that $\varphi^t$ must be dissipative in
$\Omega$, in particular, $\Omega$ is a positively invariant set and
contains a maximal attractor of $\varphi^t$
(Proposition~\ref{omega-set1}).

\medskip
\begin{Lemma}\label{tight} Let  $\cal A=\{A_\alpha\}\subset \tilde{\cal A}$ be an admissible null family
 and $\Omega\subset\cal U$ be a connected open set. Suppose that
there is a  uniform Lyapunov function in $\Omega$ with respect to
$\{{\cal L}_{A_\alpha}\}$ with  essential lower bound $\rho_m$ and
upper bound $\rho_M$.  Then the following holds.

\begin{itemize}
\item[{\rm a)}] The family $\{\mu_\alpha^\Omega\}$ of all normalized stationary measures in $\Omega$
corresponding to
 $\{{\cal L}_{A_\alpha}\}$  is $\cal A$-sequentially null compact in $M({\Omega})$;

\item[{\rm b)}] ${\rm supp} (\mu)\subset
\bar\Omega_{\rho_m}$ for any  $\mu\in {\cal M}_{\cal A}^\Omega$,
where $\Omega_{\rho_m}$ denotes the $\rho_m$-sublevel set of $U$.
\end{itemize}
\end{Lemma}
\begin{proof} Denote $\gamma$ as the uniform Lyapunov constant of
$U$ and  $\Omega_\rho$ as the $\rho$-sublevel set of $U$ in $\Omega$
for each $\rho\in [\rho_m,\rho_M)$.

 To prove a), we let $\{A_l=(a^{ij}_l)\}\subset \cal A$ be a sequence with $A_l\to 0$ as
$l\to\infty$, and denote $\mu_l:=\mu_{\ell}^\Omega$ for all $l$.
 For a
fixed $\rho_0\in (\rho_m,\rho_M)$, we have by Proposition \ref{unif} a) that
\begin{equation}\label{Est1}
\mu_l(\Omega\setminus \bar\Omega_{\rho_0})\le\mu_{l}(\Omega\setminus
\Omega_{\rho_0})\le
\gamma^{-1}C_{\rho_m,\rho_0}|A_{l}|_{C(\Omega_{\rho_0}\setminus\Omega_{\rho_m})}
|\nabla
U|^2_{C(\Omega_{\rho_0}\setminus\Omega_{\rho_m})}\mu_{l}(\Omega_{\rho_0}\setminus\Omega_{\rho_m}),
\end{equation}
for all $l=1,2,\cdots$, where $C_{\rho_m,\rho_0}$ is a constant
depending only on $\Omega,\rho_0,\rho_m$. It
follows that for any $\epsilon>0$, there is a natural number $l_0$
such that
$$\mu_l(\Omega\setminus \bar\Omega_{\rho_0})<\epsilon, \qquad  l=l_0+1,\cdots.
$$  For each $l=1,\cdots,
l_0$, clearly we can find a compact set $S_l\subset \Omega$ such
that $\mu_l(\Omega\setminus S_l)<\epsilon$. Let $K_\epsilon :=
\bar\Omega_{\rho_0}\cup (\bigcup_{l=1}^{l_0} S_l)$. Then
$\mu_l({\Omega}\setminus K_\epsilon)<\epsilon$ for all $l$. This
shows that the sequence $\{\mu_l\}$ is tight, hence relatively
sequentially compact in $M({\Omega})$.

To show b), we note that ${\Omega}\setminus \bar\Omega_{\rho_0}$ is
open. Applying Proposition \ref{wt-kh}, we have by  taking limit
$l\to\infty$ in \eqref{Est1} that
\begin{equation}\label{t3}
0=\liminf_{l\to\infty} \mu_l ({\Omega}\setminus
\bar\Omega_{\rho_0})\ge \mu({\Omega}\setminus \bar\Omega_{\rho_0}).
\end{equation}
This implies ${\rm supp} (\mu)\subset \bar\Omega_{\rho_m}$ by the
continuity of probability measures and the arbitrariness of
$\rho_0\in (\rho_m,\rho_M)$.
\end{proof}
\medskip

Note that Theorem B a) is just  Lemma \ref{tight} a) when taking
$\Omega=\cal U$.

\medskip

\begin{Theorem}\label{t410} Assume $V\in C^1(\cal U,\R^n)$. Let  $\cal A=\{A_\alpha\}\subset \tilde{\cal A}$ be
an admissible null family
 and $\Omega\subset\cal U$ be a connected open set. If there is a
uniform Lyapunov function in $\Omega$ with respect to $\{{\cal
L}_{A_\alpha}\}$, then the following holds:
\begin{itemize}
\item[{\rm  a)}] ${\cal M}^\Omega_{\cal A}\ne\emptyset$ and each $\mu\in {\cal
M}_{\cal A}^\Omega$ is an invariant measure of $\varphi^t$ in
$\Omega$ supported on  $ \cal J$ - the maximal attractor of
$\varphi^t$ in $\Omega$.
\item[{\rm b)}] If $\cal A$ is invariant,
then ${\cal J}_{\cal A}^\Omega$ and $\cal J$ are the smallest and
largest relatively $\cal A$-stable sets in $\Omega$, respectively.
\end{itemize}
\end{Theorem}

\begin{proof} a) Denote by $\{\mu_\alpha^\Omega\}$ the set of all normalized stationary measures in $\Omega$ corresponding to
 $\{{\cal L}_{A_\alpha}\}$.  By Lemma~\ref{tight} a), ${\cal M}^\Omega_{\cal A}\ne\emptyset$. Let $\mu\in
{\cal M}_{\cal A}^\Omega$. Then there is a sequence
$\{\mu_l:=\mu_{A_\ell}^\Omega\}$ such  that $\mu=\lim_{l\to\infty}
\mu_{l}$, where  $\{A_l=(a^{ij}_l)\}\subset \cal A$ is a sequence
such that $A_l\to 0$ as $l\to\infty$. Denote $U$ as a uniform
Lyapunov function in $ \Omega$ with respect to $\{{\cal
L}_{A_\alpha}\}$, $\rho_m$ as an essential lower bound of $U$,
$\rho_M$ as the essential upper bound of $U$, and $\Omega_\rho$ as
the $\rho$-sublevel set of $U$ for each $\rho>0$. By Lemma
\ref{tight} b), we have ${\rm supp}(\mu)\subseteq
\bar\Omega_{\rho_m}$.

For fixed $\rho\in(\rho_m,\rho_M)$, by Proposition~\ref{UNIF},
$\Omega_{\rho}$ is an open, pre-compact, positively invariant set of
$\varphi^t$ in $\Omega$ containing ${\rm supp}(\mu)$, so it follows
from Theorem \ref{tight0-1} c) and Proposition
\ref{invariant-measure} that $\mu$ is an invariant measure of
$\varphi^t$ supported on the maximal attractor $\cal
J=\omega(\Omega_{\rho_m})$ of $\varphi^t$ in $\Omega_{\rho_m}$,
which is also the maximal attractor of $\varphi^t$ in $\Omega$.

b) We note by the invariance of any $\mu\in {\cal M}_{\cal
A}^\Omega$ that $\cal J^\Omega_{\cal A}$ is a compact invariant set
of $\varphi^t$ in $\Omega$ and by a) that $\cal J^\Omega_{\cal
A}\subset \cal J$. It follows from Proposition \ref{stable}, Lemma
\ref{tight} a) and Corollary \ref{min-A-S} that  both $\cal
J^\Omega_{\cal A}$ and $\cal J$ are relatively $\cal A$-stable sets
in $\Omega$ with $\cal J^\Omega_{\cal A}$ being the smallest
relatively $\cal A$-stable set  in
 $\Omega$. Since the maximal attractor $\cal J$ is the largest
 compact invariant set of $\varphi^t$ in
 $\Omega$, the proof is complete.
\end{proof}

Note that Theorem~\ref{t410} a) immediately implies
Theorem~B b) when $\Omega=\cal U$.

%\medskip
%{\color{red}
%\begin{Corollary} Assume $V\in C^1(\cal U,\R^n)$. Let  $\cal A=\{A_\alpha\}\subset \tilde{\cal A}$ be a null family.
%If there is a uniform Lyapunov function in $\cal U$ with respect to
%$\{{\cal L}_{A_\alpha}\}$ and  and {\rm Theorem~\ref{t410}} holds on
%$\cal U$.
%\end{Corollary}

%\begin{proof} It follows from Lemma~\ref{tight} and
%Theorem~\ref{t410}.
%\end{proof}
%}
\medskip

\begin{Corollary}\label{LS} Assume  $V\in C^1(\cal U,\R^n)$. If \eqref{ode} admits  a $C^2$ Lyapunov function
  whose  second derivatives are bounded on $\cal U$, then the
global attractor $\cal J$ in $\cal U$ is $\cal A$-stable with
respect to any invariant, bounded null family $\cal
A$.
\end{Corollary}

\begin{proof}
It follows from Proposition~\ref{UNIF-S}, Remark~\ref{rk2.3} 1), and Theorem~\ref{t410} with $\Omega=\cal U$.
\end{proof}

%{\color{red} We note that the assumptions of Corollary \ref{LS} imply that $\varphi^t$ is a flow on $\cal U$. }

\subsection{Local concentration of limit measures} For a given admissible null family
$\mathcal{A}=\{A_\alpha\}\subset \tilde{\cal A}$, if there exists a
uniform Lyapunov  function with respect to $\{{\cal
L}_{A_\alpha}\}$, then  when $\Omega=\cal U$ Theorem~\ref{t410}
asserts that any limit measure $\mu\in {\cal M}_{\cal A}$ is an
invariant measure of $\varphi^t$ supported on the global attractor
$\cal J$ of $\varphi^t$. In this case, possible local concentration
of $\mu$ is clear because any invariant measure of a flow defined on
a compact metric space is supported on the closure of the recurrent
set of the flow (\cite{Walter}).

In case that a uniform Lyapunov function is not known on $\cal U$,
information may still be obtained about  possible local
concentration of a limit measure  $\mu\in {\cal M}_{\cal A}$  on a
compact invariant set ${\cal E}\subset\cal U$ of $\varphi^t$, i.e.,
there is a neighborhood $\cal W$ of $\cal E$ in $\cal U$ such that
$\mu(\cal W\setminus \cal E)=0$, provided that a uniform Lyapunov or
anti-Lyapunov function is known on $\cal W$.

We first consider   a pre-compact set $\Omega\subset\cal U$ in which
a uniform anti-Lyapunov function with respect to $\{{\cal
L}_{A_\alpha}\}$ exists. We note that  ${\cal M}_{\cal A}^{\bar
\Omega}$ is always non-empty (see Remark~\ref{tt}) even if ${\cal
M}_{\cal A}^\Omega$ is empty. We also note by Proposition~\ref{UNIF}
that, with the existence of a uniform anti-Lyapunov function in
$\Omega$ with respect to $\{{\cal L}_{A_\alpha}\}$, $\varphi^t$ must
be anti-dissipative in $\Omega$. In particular, $\Omega$ contains a
maximal repeller of $\varphi^t$.

\medskip

\begin{Lemma}\label{tight1} Let $\cal A=\{A_\alpha\}\subset
\tilde{\cal A}$ be a null family. Suppose that  $\Omega$ is a
connected,  open,  and pre-compact subset of $\cal U$
 and that there is a uniform anti-Lyapunov function $U$ with
respect to $\{{\cal L}_{A_\alpha}\}$ in $\Omega$.  Then the
following holds.
\begin{itemize}
\item[{\rm a)}]  The family $\{\mu_\alpha^\Omega\}$ of normalized stationary
measures in $\Omega$  corresponding to $\{{\cal L}_{A_\alpha}\}$  is
relatively sequentially compact in $M(\bar{\Omega})$.
\item[{\rm b)}] For any $\mu\in {\cal M}^{\bar\Omega}_{\cal A}$,
${\rm supp} (\mu)\subset  \bar\Omega_{\rho_m}\cup
\partial{\Omega}$, where $\rho_m$ is the essential lower bound of $U$
and $\Omega_{\rho_m}$ is the $\rho_m$-sublevel set of $U$.
\item[{\rm c)}] If $V\in C^1(\cal U,\R^n)$, then for any $\mu\in {\cal
M}^{\bar\Omega}_{\cal A}$, ${\rm supp} (\mu)\subset \cal
R\cup
\partial{\Omega}$, where $\cal R$ is the maximal repeller of
$\varphi^t$ in $\Omega$. Moreover, if $\mu(\cal R)\ne 0$, then
$\frac{\mu\big |_{\cal R}}{\mu(\cal R)}$ is an invariant measure of
$\varphi^t$ on $\cal R$, and consequently, $\mu$ is an invariant
measure of $\varphi^t$ in $\Omega$ if $\mu(\partial\Omega)=0$.
\end{itemize}
\end{Lemma}
\begin{proof} a) holds because $\{\mu_\alpha^\Omega\}$, as a family of Borel
probability measures on the compact set $\bar\Omega$, is tight.

To prove b), we let $\{A_l=(a^{ij}_l)\}\subset \cal A$ and
$\{\mu_l:=\mu_l^\Omega\}$ be  sequences such that $A_l\to 0$ and $
\mu_l\to \mu\in M(\bar {\Omega})$ as $l\to\infty$. Denote $ \rho_M$
as the essential upper bound of $U$, $\gamma$
 as the uniform anti-Lyapunov constant of $U$, and
$\Omega_\rho$ as the  $\rho$-sublevel set of $U$ in $\Omega$ for
each $\rho\in [\rho_m, \rho_M)$.

For each $\rho\in [\rho_m,\rho_M)$, we  have
$$ a^{ij}_l(x)
\partial_i U(x) \partial_j U(x)\le |A_l|_{C(\bar{\Omega})}|\nabla
U|^2_{C(\bar\Omega_{\rho})},\qquad x\in \partial \Omega_\rho.
$$
Thus, for a given $\rho_0\in (\rho_m,\rho_M)$, it follows from
Proposition \ref{aunif} that
\begin{equation}\label{anti-ineq}
\mu_{l}(\Omega_{\rho_0}\setminus \Omega_{\rho_m}^*)\le
\mu_{l}(\Omega_{\rho}\setminus
\Omega_{\rho_m}^*)\rme^{-\frac{\gamma C_\rho}{a_l}}\le
\rme^{-\frac{\gamma C_\rho}{a_l}},
\end{equation}
where $\Omega_{\rho_m}^*=\Omega_{\rho_m}\cup U^{-1}(\rho_m)$,
$a_l=|A_l|_{C(\bar{\Omega})}$, and
$C_\rho=\frac{\rho-\rho_0}{|\nabla U|^2_{C(\bar\Omega_{\rho})}}$,
$\rho\in (\rho_0,\rho_M)$.

 Since $A_l\to 0$ and ${\cal L}_{A_l} U(x)\geq \gamma$ for all $l$ and
$x\in \tilde{\Omega}=\Omega\setminus \bar\Omega_{\rho_m}$, taking
limit $l\to\infty$ yields that $V(x)\cdot\nabla U(x)\ge \gamma$ in
$\tilde \Omega$. Hence $V(x)\cdot\nabla U(x)\ge \gamma$ on the the
closure of $\tilde\Omega$ in $\Omega$, i.e., on the set
$\{x\in\Omega: U(x)\ge \rho_m\}$.  For any $\rho\in[\rho_m,\rho_M)$
and $x\in\partial \Omega_{\rho}$,  the Implicit Function Theorem
implies that in the vicinity of $x$, $\partial \Omega_\rho$ is a
local $C^2$ hypersurface which coincides with a portion of the level
surface $U^{-1}(\rho)$. Since $\partial \Omega_\rho$ is a compact
set, it consists of finitely many $C^2$ components, each of which is
oriented. Thus, $\Omega_\rho$ is a $C^2$ domain and
\begin{equation}\label{nonde}
\partial\Omega_\rho = U^{-1}(\rho), \qquad \hbox{for any } \rho\in[\rho_m,\rho_M).
\end{equation}
In particular, $\Omega_{\rho_m}^*=\bar\Omega_{\rho_m}$. Since
$a_l\to 0$, an application of Proposition~\ref{wt-kh} to
\eqref{anti-ineq} yields that $\mu(\Omega_{\rho_0}\setminus
\bar\Omega_{\rho_m})=0$ for any $\rho_0\in (\rho_m,\rho_M)$. It
follows that $\mu(\Omega\setminus \bar\Omega_{\rho_m})=0$. Hence
${\rm supp}(\mu)\subset \bar\Omega_{\rho_m}\cup
\partial{\Omega}$.

To prove c), we observe that if $\mu(\bar\Omega_{\rho_m})=0$, then
by b),  ${\rm supp}(\mu)\subset
\partial{\Omega}$. Now suppose  $\mu(\bar\Omega_{\rho_m})\ne 0$ and let $\rho_*,\rho_0\in (\rho_m,\rho_M)$
be fixed such that $\rho_*<\rho_0$. We have by b) that
$\mu(\partial\Omega_{\rho_0})=0$. It follows from Proposition
\ref{wt-kh} that
$\lim_{l\to\infty}\mu_l(\Omega_{\rho_0})=\mu(\Omega_{\rho_0})$, and
therefore
$$\tilde\mu_l=:\frac
{\mu_l|_{\Omega_{\rho_0}}}{\mu_l(\Omega_{\rho_0})}
$$
is a sequence of stationary measures corresponding to $\{{\cal
L}_{A_l}\}$ in $\Omega_{\rho_0}$ which converges to
$$\tilde\mu=:\frac
{\mu|_{\Omega_{\rho_0}}}{\mu(\Omega_{\rho_0})}
$$
as $l\to\infty$.  We clearly have ${\rm supp}(\tilde\mu)\subset
\Omega_{\rho_*}$ because $\rho_m < \rho_*$. Since $V(x)\cdot\nabla
U(x)\ge \gamma$ in $\{x\in\Omega: U(x)\ge \rho_m\}$,
$\Omega_{\rho_*}$ is a negatively invariant, pre-compact open subset
of $\Omega_{\rho_0}$. It follows from  Theorem \ref{tight0-1} c)
with $\Omega$ in place of $\cal U$ and $\Omega_{\rho_0}$ in place of
$\Omega$ that $\tilde\mu$ is an invariant measure of $\varphi^t$ on
$\Omega_{\rho_0}$. Moreover, by Proposition~\ref{invariant-measure},
$\tilde\mu$ is supported on $\alpha(\Omega_{\rho_0})=\cal R$ - the
maximal repeller of $\varphi^t$ in $\Omega$. Therefore, ${\rm supp}
(\tilde\mu) \subset \cal R$. By b), ${\rm supp} (\mu) \subset \cal
R\cup
\partial{\Omega}$. Since $\frac{\mu\big |_{\cal R}}{\mu(\cal R)}=\tilde\mu\big |_{\cal
R}$, the proof is complete.
\end{proof}
\medskip

\begin{Theorem}\label{anti-L}  Assume $V\in C^1(\cal U,\R^n)$. Let $\cal A=\{A_\alpha\}\subset
\tilde{\cal A}$ be an admissible null family, $\mu\in {\cal M}_{\cal
A}$, and $\Omega\subset\cal U$ be a connected open set. Then the
following holds.
\begin{itemize}
\item[{\rm a)}] If there exists a uniform Lyapunov function in $\Omega$  with respect
to $\{{\cal L}_{A_\alpha}\}$, then $\mu(\Omega\setminus\cal J)=0$,
where $\cal J$ is the maximal attractor of $\varphi^t$ in $\Omega$.
\item[{\rm b)}] If $\Omega$ is  pre-compact in $\cal U$ and there exists a uniform anti-Lyapunov
function in $\Omega$  with respect to $\{{\cal L}_{A_\alpha}\}$,
then $\mu(\Omega\setminus \cal R)=0$,  where $\cal R$ is the maximal
repeller of $\varphi^t$ in $\Omega$.
\end{itemize}
\end{Theorem}

\begin{proof} Let $\{A_l=(a^{ij}_l)\}\subset \cal A$
 be a sequence such that $A_l\to
0$ and $\{\mu_l\}$ be a sequence of stationary measures
 corresponding to $\{{\cal L}_{A_l}\}$ such that $ \mu_l\to
\mu\in M(\cal U)$ as $l\to\infty$.   Also let $U$ be a uniform
Lyapunov function in $\Omega$ in the case a) and a uniform
anti-Lyapunov function in $\Omega$ in the case b), with essential
lower bound $\rho_m$ and upper bound $\rho_M$. For each $\rho\in
[\rho_m,\rho_M)$, denote $\Omega_\rho$ as the $\rho$-sublevel set of
$U$.

If $\mu(\Omega)=0$, then the theorem holds automatically. We now
suppose that  $\mu(\Omega)\ne 0$. For any $\rho_0\in
(\rho_m,\rho_M)$, applications of \eqref{Est1}, \eqref{anti-ineq},
\eqref{nonde} to
\[
\tilde\mu_l=:\displaystyle \frac
{\mu_{l}|_{\Omega}}{\mu_{l}(\Omega)}
\]
in place of $\mu_l$ for the case a), b) respectively yield that
\begin{eqnarray}
&& \mu(\Omega\setminus\bar\Omega_{\rho_0})=0,\;\hbox{in the case
a)},\label{case1}\\
&&\mu(\Omega_{\rho_0}\setminus\bar\Omega_{\rho_m})}=0,\;\hbox{in the
case b),\label{case2}
\end{eqnarray}
for any $\rho_0\in (\rho_m,\rho_M)$. Now let $\rho_*\in
(\rho_m,\rho_M)$ be fixed. Then  \eqref{case1} and \eqref{case2}
imply that
\[
 \mu(\partial \Omega_{\rho_*})=0.
\]
Hence it follows from Proposition \ref{wt-kh} that
\[
\mu_l^*=:\displaystyle \frac
{\mu_{l}|_{\Omega_{\rho_*}}}{\mu_{l}(\Omega_{\rho_*})}
\]
is a sequence of stationary measures corresponding to $\{{\cal
L}_{A_l}\}$ in $\Omega_{\rho_*}$ which converges to
\[
\mu^*=:\displaystyle \frac
{\mu|_{\Omega_{\rho_*}}}{\mu(\Omega_{\rho_*})}.
\]
We note that  $U$ is also a uniform Lyapunov function  in
$\Omega_{\rho_*}$ in the case a) and a uniform anti-Lyapunov
function in $\Omega_{\rho_*}$ in the case b), with respect to
$\{{\cal L}_{A_\alpha}\}$.

In the case a), it follows from Theorem~\ref{t410} with $\Omega$ in
place of $\cal U$ and $\Omega_{\rho_*}$ in place of $\Omega$ that
$\mu^*(\Omega_{\rho_*}\setminus \cal J)=0$, implying that
$\mu(\Omega_{\rho_*}\setminus \cal J)=0$. In the case b), since
$\mu^*(\partial \Omega_{\rho_*})=0$, we have by Lemma~\ref{tight1}
c)  with $\Omega$ in place of $\cal U$ and $\Omega_{\rho_*}$ in
place of $\Omega$ that $\mu^*(\Omega_{\rho_*}\setminus \cal R)=0$,
implying that $\mu(\Omega_{\rho_*}\setminus \cal R)=0$. This
completes the proof because $\rho_*\in  (\rho_m,\rho_M)$ is
arbitrary.
\end{proof}
\medskip

 Part a) of the following result is precisely part c) of
Theorem~A.
\medskip

\begin{Corollary}~\label{LS2}  Assume $V\in C^1(\cal U,\R^n)$ and
$\varphi^t$ admits a strong local attractor $\cal J_0$ {\rm (}resp.
strong local repeller $\cal R_0${\rm )} with an isolating
neighborhood ${\cal W}_{{\cal J}_0}$ {\rm (}resp. ${\cal W}_{\cal
R_0}${\rm )}. Then the following holds for any admissible null
family  $\cal A=\{A_\alpha\}\subset \tilde{\cal A}$.
\begin{itemize}
\item[{\rm a)}] For any $\mu\in {\cal M}_{\cal A}$,    $\mu({\cal
W}_{\cal J_0}\setminus \cal J_0)=0$ {\rm (}resp. $\mu({\cal W}_{\cal
R_0}\setminus \cal R_0)=0${\rm )}.
\item[{\rm b)}] $\cal J_0$ is relatively $\cal A$-stable in ${\cal
W}_{\cal J_0}$ if $\cal A$ is invariant.
\end{itemize}
\end{Corollary}
\begin{proof}  We note by the proof of
Proposition~\ref{attractor-repeller} that \eqref{ode} admits a $C^2$
Lyapunov (resp. anti-Lyapunov) function $U$ in $\cal W=:{\cal
W}_{\cal J_0},{\cal W}_{\cal R_0}$.  Since $\cal W$ is a pre-compact
subset of $\cal U$, all second derivatives of $U$ are bounded in
$\cal W$. Let $\cal A=\{A_\alpha\}\subset \tilde{\cal A}$ be any
null family. Without loss of generality, we assume $d(A_\alpha,0)\ll
1$. It follows from Proposition~\ref{UNIF-S}  that $U$
 is a uniform Lyapunov (resp. anti-Lyapunov) function in $\cal W$  with respect to
 $\{{\cal L}_{A_\alpha}\}$.  By Theorem~\ref{anti-L},
$\mu(\cal W\setminus \cal J_0)=0$ {\rm (}resp. $\mu(\cal W\setminus
\cal R_0)=0${\rm )}. This proves a).

Using the uniform Lyapunov function $U$ in ${\cal W}_{\cal J_0}$
with respect to
 $\{{\cal L}_{A_\alpha}\}$ and the fact that $\cal J_0$ is the maximal attractor of $\varphi^t$ in ${\cal W}_{\cal J_0}$, b) follows from
Theorem~\ref{t410} with $\Omega={\cal W}_{\cal J_0}$.
\end{proof}

\section{Stabilization and  de-stabilization via multiplicative noises}

Throughout the section, we assume  $V\in C^1(\cal U,\R^n)$ and that
\eqref{ode} admits a $C^2$ Lyapunov function with bounded second
derivatives on $\cal U$. For any bounded null family $\cal
A=\{A_\alpha\}\subset \tilde {\cal A}$, it follows from
Proposition~\ref{UNIF-S}, Remark~\ref{rk2.3} 1), and
Theorem~\ref{t410} that $\cal A$ is admissible, $\cal
M_{\cal A}\neq\emptyset$,  and any limit measure $\mu\in {\cal
M}_{\cal A}$ is an invariant measure of $\varphi^t$ supported on the
global attractor $\cal J$ of $\varphi^t$, i.e., $\cal
J$ is $\cal A$-stable if $\cal A$ is also invariant. Moreover, by
Corollary~\ref{LS2}, if $\cal J$ contains a strong local attractor
$\cal J_0$ (resp. repeller $\cal R_0$), then $\mu$ may be locally
concentrated on $\cal J_0$ (resp. $\cal R_0$). In fact, we believe
that local concentration on a strong local repeller can be expected
only when it further contains a sub-local attractor, as suggested by
Theorem~\ref{equili} below.

In this section, we will demonstrate  an important role played by
multiplicative noise perturbations. We will show that if $\cal J$
contains a strong local attractor  $\cal J_0$ of $\varphi^t$, then
one can design a particular  invariant,  bounded null family $\cal
A=:\{A_\alpha\}\subset \tilde {\cal A}$ by adding stronger noise
away from the local attractor such that any $\mu\in {\cal M}_{\cal
A}$ is (globally) concentrated on $\cal J_0$, i.e., $\cal J_0$ is
$\cal A$-stable. On the contrary, if $\cal J$ contains a strong
local repeller $\cal R_0$ of $\varphi^t$, then one can design a
particular  invariant,  bounded null family $\cal
A=:\{A_\alpha\}\subset \tilde{\cal A}$ by adding stronger noise on
and near the repeller such that any $\mu\in {\cal M}_{\cal A}$ is
(globally) concentrated away from $\cal R_0$,  i.e., $\cal R_0$ is
strongly $\cal A$-unstable.

In the case that $\cal R_0$ is a strongly repelling equilibrium, we
will show that it is strongly $\cal A$-unstable with respect to any
invariant,  normal null family $\cal A\subset\tilde{\cal A}$.
\medskip

\subsection{Stabilizing a local attractor and de-stabilizing a local repeller} Let ${\cal E}_0$ be
either a strong local attractor or a strong local repeller of
$\varphi^t$ in $\cal U$. We denote the isolating  neighborhood of
${\cal E}_0$  by ${\cal W}_0$. Then by the proof of
Proposition~\ref{attractor-repeller} there is a positive function
$U_0\in C^2(\cal U)$ and constants $\gamma_0,\tilde{\rho}>0$ such
that $\partial{\cal W}_0=\{U_0(x)=\tilde{\rho}\}$ and
\begin{equation}\label{t4}
|V(x)\cdot \nabla U_0(x)|>\gamma_0 |\nabla U_0(x)|,\qquad\qquad x\in
\partial{\cal W}_0=\partial{\Omega^0_{\tilde\rho}},
\end{equation}
where $\Omega^0_\rho$ denotes the  $\rho$-sublevel set of $U_0$ for
each $\rho\ge 0$. By modifying $U_0$ away from  $\partial{\cal
W}_0$, we can assume without loss of generality that $U_0(x)<\tilde
\rho$ for all $x\in {\cal W}_0$, so that $U_0$ becomes a compact
function in ${\cal W}_0$. We note that $\bar {\cal
W}_0=\bar\Omega^0_{\tilde\rho}=\Omega^{0}_{\tilde \rho}\cup
U_0^{-1}(\tilde \rho)$.
\medskip

\begin{Lemma}\label{SAR} Let ${\cal E}_0$, ${\cal W}_0$, $U_0$,
$\tilde \rho$, $\gamma_0$, $\Omega^0_\rho$ be as in the above, and
$\cal A=\{A_\alpha\}\subset \tilde{\cal A}$ be a null family on
$\cal U$. Then there are positive constants $\rho_*,\rho^*$,
depending only on ${\cal W}_0, U_0$, with $\rho_*< \tilde \rho<
\rho^*$, such that as $d(A_\alpha,0) \ll 1$,
\begin{eqnarray*}
&& \mu_\alpha(\Omega^0_{\rho^*}\setminus {\cal
W}_0)=\mu_\alpha(\bar\Omega^0_{\rho^*}\setminus {\cal W}_0)\le
\rme^{-\frac{C_1}{a_1(\alpha)}}, \qquad\quad \hbox{when ${\cal E}_0$
is a
local attractor},\\
&&\mu_\alpha({\cal W}_0\setminus
\bar\Omega^0_{\rho_*})=\mu_\alpha(\bar{\cal W}_0\setminus
\Omega^0_{\rho_*})\le \rme^{-\frac{C_2}{a_2(\alpha)}}, \qquad\quad
\hbox{when ${\cal E}_0$ is a local repeller}
\end{eqnarray*}
for all  stationary measures $\{\mu_\alpha\}$ corresponding to
$\{{\cal L}_{A_\alpha}\}$, where
\begin{eqnarray*}
&& C_1= \frac{\gamma_0(\tilde\rho-\rho_*)\min_{x\in
\partial{\Omega^0_{\tilde\rho}}}|\nabla U_0(x)|}{ 2\max_{\rho_*\le U_0(x)\le \tilde\rho}
|\nabla U_0(x)|^2},\\
&& a_1(\alpha)=\max_{\rho_*\le U_0(x)\le \tilde\rho} |A_\alpha(x)|,\\
&& C_2= \frac{\gamma_0(\rho^*-\tilde\rho)\min_{x\in
\partial{\Omega^0_{\tilde\rho}}}|\nabla U_0(x)|}{2\max_{\tilde\rho\le U_0(x)\le \rho^*} |\nabla
U_0(x)|^2},\\
&& a_2(\alpha)= \max_{\tilde\rho\le U_0(x)\le \rho^*} |A_\alpha(x)|.
\end{eqnarray*}
\end{Lemma}

\begin{proof} By \eqref{t4},
there are $\rho_*,\rho^*$ with $\rho_*<\tilde\rho<\rho^*$ such that,
for all $ x\in \bar\Omega^0_{\rho^*}\setminus \Omega^0_{\rho_*}$,
$\nabla U_0(x)\neq 0$ and
\[
|{\cal L}_{A_\alpha} U_0(x)|= |a^{ij}_\alpha\partial^2_{ij} U_0(x)
+V(x)\cdot \nabla U_0(x)|\ge  \frac{\gamma_0 \min_{x\in
\partial{\Omega^0_{\tilde\rho}}} |\nabla U_0(x)|}2,
\]
as $d(A_\alpha,0)\ll 1$. Since $\mathcal{W}_0$ is connected, by
making $\rho_*,\rho^*$ sufficiently close to $\tilde\rho$ if
necessary, we may assume without loss of generality that
$\Omega_\rho^0$ is a connected open set for each $\rho\in
(\rho_*,\rho^*)$.  As in the proof of Lemma \ref{tight1}, the
Implicit Function Theorem implies that $\Omega^0_\rho$ is a $C^2$
domain and $\bar\Omega^0_\rho =\Omega_\rho^{0}\cup U_0^{-1}(\rho)$
for each $\rho\in (\rho_*,\rho^*)$.

%For each $\alpha$, since $\mu_\alpha(\partial{\cal
%W}_0)=\mu_\alpha(\partial{\Omega}^0)=0$, we  have
% \[
% \mu_\alpha(\Omega^0_{\rho^*}\setminus \bar{\cal
%W}_0)=\mu_\alpha(\bar\Omega^0_{\rho^*}\setminus {\cal
%W}_0),\qquad
% \mu_\alpha({\cal W}_0\setminus
%\bar\Omega^0_{\rho_*})=\mu_\alpha(\bar{\cal W}_0\setminus
%\Omega^0_{\rho_*}).
%\]

In the case that ${\cal E}_0$ is a strong local attractor, the
remaining part of the lemma follows from Proposition \ref{unif} b)
with $\Omega_{\rho^*}^0$ in place of $\Omega$ and
$$\rho_M=:\rho^*,\;\;
\rho_m=\rho_*, \;\; \rho=:\tilde\rho, \; \;\gamma=:\frac{\gamma_0
\min_{x\in
\partial{\Omega^0_{\tilde\rho}}} |\nabla U_0(x)|}2, \;\;
H=a_1(\alpha)\max_{\rho_*\le U_0(x)\le \tilde\rho} |\nabla
U_0(x)|^2.
$$

In the case that ${\cal E}_0$ is a strong local repeller, the
remaining part of the  lemma follows from  Proposition \ref{aunif}
with ${\cal W}_0$ in place of $\Omega$ and
$$\rho=:\rho^*,\;
\rho_m=\rho_*, \;\; \rho_0=:\tilde\rho,\; \; \gamma=:\frac{\gamma_0
\min_{x\in
\partial{\Omega^0_{\tilde\rho}}} |\nabla U_0(x)|}2, \;\; H=a_2(\alpha)\max_{\tilde\rho\le U_0(x)\le \rho^*} |\nabla
U_0(x)|^2.
$$
\end{proof}
\medskip

Theorem~C a), b) follow from the following result.

\medskip

\begin{Theorem}\label{thm43} Assume  $V\in C^1(\cal U,\R^n)$ and that
\eqref{ode} admits a $C^2$  Lyapunov function with bounded second
derivatives on $\cal U$. Then the following holds.
\begin{itemize}

\item[{\rm a)}]
For any strong local attractor  ${\cal J}_0$ lying in $\cal U$,
there is a normal null family $\cal A\subset \tilde{\cal A}$ on
$\cal U$ which is both invariant and admissible such
that $\cal M_{\cal A}\neq \emptyset$ and each limit measure $\mu\in
{\cal M}_{\cal A}$ is an invariant measure of $\varphi^t$ supported
on ${\cal J}_0$. Consequently, ${\cal J}_0$ is $\cal A$-stable.
\item[{\rm b)}]
For any strong local  repeller ${\cal R}_0$ lying in $\cal U$, there
is a normal null family $\cal A\subset \tilde{\cal A}$
which is both invariant and admissible such that  $\cal M_{\cal A}\neq \emptyset$ and each limit measure $\mu\in
{\cal M}_{\cal A}$ is an invariant measure of $\varphi^t$ supported
 on the maximal attractor ${\cal J}_*$ of $\varphi^t$ in
 $\cal J\setminus {\cal W}_0$, where  $\cal J$ is  the global attractor  of
$\varphi^t$ and ${\cal W}_0$ is an isolating neighborhood of ${\cal
R}_0$. Consequently, ${\cal J}_*$ is $\cal A$-stable and ${\cal
R}_0$ is strongly $\cal A$-unstable.
\end{itemize}
\end{Theorem}

\begin{proof} Let ${\cal E}_0$ be either ${\cal J}_0$ or ${\cal R}_0$,
and ${\cal W}_0$ be an isolating neighborhood of ${\cal E}_0$. We
note that since $\cal U$ is dissipative, ${\cal E}_0$ must lie in
the global attractor ${\cal J}$ of $\varphi^t$.

First of all, we fix a pre-compact,  connected open neighborhood
$\Omega_0$ of $\cal J$ in $\cal U$. Let $\cal
A=\{A_\alpha\}=\{(a^{ij}_\alpha)\}\subset \tilde{\cal A}$ be any
bounded null family. We have by Proposition~\ref{UNIF-S} that there
is a uniform Lyapunov function with respect to $\{{\cal
L}_{A_\alpha}\}$ in $\cal U$. Hence by
Remark~\ref{rk2.3} 1) $\cal A$ is admissible. Denote
 $\{\mu_\alpha\}$ as the set of all   stationary
measures  corresponding to $\{{\cal L}_{A_\alpha}\}$. It follows
from Lemma~\ref{tight} a) that $\{\mu_\alpha\}$ is $\cal
A$-sequentially null compact in $M(\cal U)$ and hence $\cal M_{\cal
A}\neq\emptyset$. By Theorem~\ref{t410}, any $\mu\in \cal M_{\cal
A}$ is an invariant measure of $\varphi^t$ supported on the global
attractor $\cal J$.

If $\cal U=\R^n$, then $\cal A$ is clearly invariant. Otherwise, one
can modify each $A\in \cal A$ near $\partial{\cal U}$ such that
$A(x)\to 0$ as $x\to \partial{\cal U}$. Since \eqref{ode} admits a
Lyapunov function in $\cal U$, it is not hard to see that the
modified normal null family $\cal A$ becomes invariant.

Let $\rho_*$, $\rho^*$, $\tilde\rho$, $\Omega^0_\rho$,
$\rho\in [\rho_*,\rho^*]$, be  as in Lemma~\ref{SAR}.  Recall that
$\rho_*<\tilde\rho<\rho^*$ and $\bar{\cal W}_0=\bar\Omega^0_{\tilde
\rho}$. By taking $\rho^*$ further small if necessary, we can assume
 $\bar\Omega^0_{\rho^*}\subset \Omega_0$.

For fixed $\tilde{\rho}^*\in (\tilde\rho,\rho^*)$,
$\tilde{\rho}_*\in (\rho_*,\tilde\rho)$, we consider sets
\begin{eqnarray*}
D&=&\left\{\begin{array}{ll} \bar\Omega_{0}\setminus
\Omega^0_{{\tilde\rho}^*},&\hbox{when ${\cal E}_0$ is a local
attractor},\\
 \bar\Omega^0_{{\tilde \rho}_*},& \hbox{when ${\cal E}_0$ is a local
repeller},
\end{array}\right.\\
D^*&=&\left\{\begin{array}{ll} \bar\Omega^0_{\rho^*}\setminus
\Omega^0_{\tilde\rho},& \hbox{when ${\cal E}_0$ is a local
attractor},\\
 \bar\Omega^0_{\tilde \rho}\setminus
\Omega^0_{\rho_*},& \hbox{when ${\cal E}_0$ is a local
repeller}, \end{array}\right.\\
D_*&=&\left\{\begin{array}{ll} \bar\Omega^0_{\tilde\rho}\setminus
\Omega^0_{\rho_*},& \hbox{when ${\cal E}_0$ is a local
attractor},\\
 \bar\Omega^0_{\rho^*}\setminus
\Omega^0_{\tilde\rho},& \hbox{when ${\cal E}_0$ is a local
repeller}.
\end{array}\right.
\end{eqnarray*}

Let  $\Omega$ be a fixed small  neighborhood   of $D$ in $\cal U$
with $\bar\Omega\cap D_*=\emptyset$. Then there are small balls
$\{B_R^k(y_k)\}_{k=1}^N$ of radius $R$ centered at $y_k\in D$,
$k=1,\cdots,N$, whose union  covers $D$, and,
$$\bigcup_{k=1}^NB_{4R}^k(y_k)\subset \Omega.$$ Using the facts that $\bar\Omega^0_{\rho^*}\subset \Omega_0$,
$\Omega_0$ is connected, and $\Omega^0_\rho$ is a connected open set
with oriented $C^2$ boundary for each $\rho\in (\rho_*,\rho^*)$, we
see that $D$ has finitely many components, say,
$D_1,D_2,\cdots,D_I$. Since $\partial \Omega^0_{{\tilde \rho}^*}
\subset \text{int}(\bar\Omega^0_{\rho^*}\setminus
\Omega^0_{\tilde\rho})$ and $\partial \Omega^0_{{\tilde \rho}_*}
\subset
\text{int}(\Omega^0_{\tilde\rho}\setminus\bar\Omega^0_{\rho_*} )$,
it is not hard to see that $|D_i\cap D^*|>0$ for all
$i=1,2,\cdots,I$.

We now restrict $\{A_\alpha\}$ to be a normal null
family satisfying $A_\alpha\in C^1(\cal U,GL(n,\R))$ for each
$\alpha$ and
\begin{equation}\label{restr1}
\sup_\alpha \tilde\Lambda_\alpha(\Omega)<1,
\end{equation}
 where, for each $\alpha$,
\[
\tilde \Lambda_\alpha(\Omega)=\sup_{x\in \Omega}|\partial
A_\alpha(x)|.
\]
For each  $\alpha$, denote $u_\alpha$ as the weak stationary
solution corresponding to ${\cal L}_{A_\alpha}$ that is associated
with  the measure $\mu_\alpha$. We also let $\Lambda_\alpha,
\lambda_\alpha$ be the quantities defined as in
Definition~\ref{nullD} on each pre-compact open subset of $\cal U$
for the restricted family $\{A_\alpha\}$.

 Note that the above neighborhood $\Omega$ of $D$ can be chosen
such that the number of components of $\Omega$ is at most $I$.
Therefore,  applying Proposition~\ref{Harn} in $\Omega$ with
$b^i=\partial_j a^{ij}_\alpha-V^i$, $c^i=d\equiv 0$, $i=1,\cdots,n$,
we have
\[
\sup_{B_R(y_i)}u_{\alpha}(x)\le C_0^{(C_1+\frac {C_2
R}{\lambda_\alpha(\Omega)})}\inf_{B_R(y_i)} u_{\alpha}(x),\quad
i=1,\cdots,N,
\]
where $C_0\ge 1$ is a positive constant depending only on $n$ and
$\Omega$, $C_1=\sup_\alpha
\frac{\Lambda_\alpha(\Omega)}{\lambda_\alpha(\Omega)}<\infty$,
 and
$C_2=2n(\sup_\alpha\tilde\Lambda_\alpha(\Omega)+\sup_{x\in
\Omega}|V(x)|)\le 2n(1+\sup_{x\in \Omega}|V(x)|)<\infty$.

Let $C_3=C_0^{NC_1}$ and $C_*=NC_2R\ln C_0$. Then for each
 $i=1,2,\cdots,I$,
\begin{eqnarray}
\sup_{D_i}u_{\alpha}(x)&\le& C_3 \rme^{\frac
{C_*}{\lambda_\alpha(\Omega)}}\inf_{D_i} u_{\alpha}(x)\le C_3
\rme^{\frac {C_*}{\lambda_\alpha(\Omega)}}\inf_{D_i\cap D^*}
u_{\alpha}(x)\nonumber\\
&\le& \frac{C_3}{|D_i\cap D^*|} \rme^{\frac
{C_*}{\lambda_\alpha(\Omega)}}\mu_\alpha(D_i\cap D^*)\le
\frac{C_3}{|D_i\cap D^*|} \rme^{\frac
{C_*}{\lambda_\alpha(\Omega)}}\mu_\alpha(D^*).\label{u-est}
\end{eqnarray}

By Lemma~\ref{SAR}, there is a positive constant $C^*$ depending
only on $U_0$ and $D_*$ such that
\begin{equation}\label{DD}
\mu_\alpha(D^*)\le \rme^{-\frac{C^*}{\Lambda_\alpha(D_*)}},
\end{equation}
as $|A_\alpha|\ll 1$.

It now follows from \eqref{u-est} and \eqref{DD} that
\[
\mu_\alpha(D)\le \sum \limits_{i=1}^I |D_i|\sup_{D_i}
u_\alpha(x)\le \sum \limits_{i=1}^I
\frac{|D_i|C_3}{|D_i\cap D^*|} \rme^{-\frac
{C^*}{\Lambda_\alpha(D_*)}+\frac {C_*}{\lambda_\alpha(\Omega)}}.
\]
Therefore, if we further restrict $\{A_\alpha\}$ to be such that
\begin{equation}\label{restr2}
\frac{ \lambda_\alpha(\Omega)}{\Lambda_\alpha(D_*)}> \frac{C_*+1}{C^*}
\end{equation}
for all $\alpha$, then $\mu_\alpha(D)\to 0$ as $A_\alpha\to 0$.
Since $D=\bar\Omega_0\setminus\Omega^0_{\tilde\rho^*}$ when $\cal
E_0$ is a local attractor  and $D=\bar\Omega^0_{\tilde\rho_*}$ when
$\cal E_0$ is a local repeller, it follows  that
\begin{equation}\label{Din}
\left\{
\begin{array}{ll}
  \mu_\alpha(\Omega_0\setminus\bar\Omega^0_{\tilde\rho^*})\to 0, & \hbox{when }\cal E_0 \hbox{ is a local attractor}, \\
  \mu_\alpha(\Omega^0_{\tilde\rho_*})\to 0, &  \hbox{when }\cal E_0 \hbox{ is a local
  repeller},
\end{array}\right.
\end{equation}
 as $A_\alpha\to 0$.
As $\Omega$ and $D_*$ are separated by a positive distance,  it is
easy to construct a null families, still denoted by $\cal A$, such
that both \eqref{restr1} and \eqref{restr2} hold simultaneously.

Let $\mu\in {\cal M}_{\cal A}$. Then it follows from Proposition
\ref{wt-kh} and \eqref{Din} that $\mu$ is supported on
$\bar\Omega^0_{\tilde\rho^*}$ in the case a) and on
$\Omega_0\setminus \Omega^0_{\tilde\rho_*}$ in the case b), both of
which are positively invariant set of $\varphi^t$.  It follows from
Proposition \ref{invariant-measure} that $\mu$ is actually supported
on $\omega(\bar\Omega^0_{\tilde\rho^*})$ which equals ${\cal J}_0$
in the case a) and on $\omega(\Omega_0\setminus
\Omega^0_{\tilde\rho_*})$ which equals ${\cal J}_*$ in the case b).
By Proposition~\ref{stable}, both ${\cal J}_0$ and ${\cal J}_*$ are
$\cal A$-stable.
\end{proof}
\medskip

The above theorem naturally applies to families of the form
$\{\epsilon A\}\subset \cal A$, where $\epsilon>0$ is a small
parameter and $A$ is a bounded, everywhere positive definite,
$n\times n$ matrix-valued $C^1$ function on $\cal U$.

\begin{Corollary}
Assume  $V\in C^1(\cal U,\R^n)$ and that \eqref{ode} admits a $C^2$
Lyapunov function with bounded second derivatives on $\cal U$. Then
there are bounded, everywhere positive definite, $n\times n$
matrix-valued $C^1$ functions $A_i$, $i=1,2$, on $\cal U$,  such
that {\rm a)}, {\rm b)} of {\rm Theorem~\ref{thm43}} hold
respectively for the family $\{\epsilon A_i\}$, $i=1,2$, as
$\epsilon\to 0$.
\end{Corollary}

\begin{proof}
For any bounded, everywhere positive definite, $n\times n$
matrix-valued $C^1$ function $A$ on $\cal U$, $\{\epsilon A\}$ is
clearly a normal null family.  In the case that $\cal
U\ne \R^n$, the invariance of $\{\epsilon A\}$ is guaranteed if
$A(x)\to 0$ as $x\to \partial{\cal U}$. The corollary now follows
from Theorem~\ref{thm43}.

In fact, by the proof of Theorem~\ref{thm43}, we just need to choose
$A=(a^{ij})$ such that
$$
\frac{\inf_\Omega \lambda_A (x)}{\max_{D_*}\Lambda_A(x)}>\frac{C_*+1}{C^*},
$$
where $\Omega, D_*, C_*,C^*$ are as in the proof of
Theorem~\ref{thm43}, $\Lambda_A(x)=\sqrt{\sum_{i,j} |a^{ij}(x)|^2}$,
and $\lambda_A(x)$ is the smallest eigenvalue of $A(x)$ for each
$x\in \cal U$.
\end{proof}

\subsection{Uniform de-stabilization of a locally repelling
equilibrium} We note  that if ${\cal R}_0$ is a strong local
repeller, then in general not every normal null family $\cal
A\subset \tilde{\cal A}$  can de-stabilize ${\cal R}_0$. This is
because ${\cal R}_0$ may be further decomposed to contain a strong
local attractor for which Theorem~\ref{thm43} a) is applicable.

We now show that if  ${\cal R}_0$ is a strongly repelling
equilibrium (see Definition~\ref{D220}), then it is de-stabilized by
any normal null family $\cal A\subset \tilde{\cal A}$. This
particularly implies part c) of Theorem~C.

\begin{Theorem}\label{equili}  Assume  $V\in C^1(\cal U,\R^n)$ and that
\eqref{ode} admits a $C^2$  Lyapunov function with bounded second
derivatives on $\cal U$. Let $x_0$ be a strongly repelling
equilibrium of $\varphi^t$ contained in the global attractor $\cal
J$ of $\varphi^t$ and ${\cal W}$ be an  isolating neighborhood  of
$x_0$. Then with respect to any  normal null family $\cal
A=\{A_\alpha\}\subset \tilde{\cal A}$,  $\cal M_{\cal
A}\neq\emptyset$, and moreover, each $\mu\in {\cal M}_{\cal A}$ is
supported on the maximal attractor ${\cal J}_*$ of $\varphi^t$ in
$\cal J\setminus \cal W$. Consequently,  with respect to any
invariant, normal null family $\cal A$, ${\cal J}_*$ is $\cal
A$-stable and $\{x_0\}$ is strongly $\cal A$-unstable.
\end{Theorem}

\begin{proof}  Let $\cal
A=\{A_\alpha\}=\{(a^{ij}_\alpha)\}\subset \tilde{\cal A}$ be any
bounded null family such that $|A_\alpha|\ll 1$. By
Proposition~\ref{UNIF-S}, $\{{\cal L}_{A_\alpha}\}$ admits a uniform
Lyapunov function in $\cal U$.   By Proposition~\ref{exist}, for
each $\alpha$, there exists a stationary measure  $\mu_\alpha$
corresponding to $\{{\cal L}_{A_\alpha}\}$  which is also regular,
i.e., $\rmd\mu_\alpha( x)=u_{\alpha}(x)\rmd x $ for some stationary
solution $u_{\alpha}\in W^{1,{p}}_{loc}(\cal U)$ corresponding to
${\cal L}_{A_\alpha}$ in $\cal U$.

Without loss of generality, we assume  $x_0=0$. In virtue of
Definition~\ref{D220}, let $U$ be an entire weak anti-Lyapunov
function associated with $\cal W$ that satisfies
\eqref{P1}-\eqref{P3}. We note  that since ${\cal W}\subset \cal J$,
it is a pre-compact set in $\cal U$. For each $\rho\in [0,\rho_M)$,
where $\rho_M$ is the essential upper bound of $U$, we denote
$\Omega_\rho$ as the $\rho$-sublevel set of $U$.  We note by
\eqref{P3} that $\partial\Omega_\rho=U^{-1}(\rho)$ and hence
$\bar\Omega_{\rho}=\{x\in\cal W: U(x)\le \rho\}$ for any
$\rho\in(0,\rho_M)$.

Fix $\bar\rho\in(0,\rho_M)$. Let $\lambda_\alpha^{\bar\rho}
:=\inf_{x\in {\Omega_{\bar\rho}}} \lambda_\alpha(x)$,
 $\lambda_U^{\bar\rho} :=\inf_{x\in {\Omega_{\bar\rho}}}\lambda_U(x)$, where
 $\lambda_\alpha(x)$, $\lambda_U(x)$ denote the smallest eigenvalues
 of $A_\alpha(x)$, $D^2 U(x)$, respectively. Then for each
 parameter $\alpha$,
  \begin{align*}
 {\cal L}_{ A_\alpha} U(x)&=a^{ij}_\alpha(x)\partial^2_{ij}U(x)+V(x)\cdot\nabla
 U(x)\\
 &\ge a^{ij}_\alpha(x)\partial^2_{ij}U(x)=\text{tr}\Big(A_\alpha(x)D^2U(x)\Big)
 \ge \lambda_U^{\bar\rho} \lambda_\alpha^{\bar\rho} =:\gamma_\alpha^{\bar\rho},\qquad  x\in
\Omega_{\bar\rho},
\end{align*}
i.e., $U$ is an anti-Lyapunov function with respect to $\{{\cal
L}_{A_\alpha}\}$ on $\Omega_{\bar{\rho}}$ with the essential lower
bound $\rho_m=0$.

 By \eqref{P1}, \eqref{P2} and the Taylor's expansion of $U$ at the
point $0$, it is easy to see that there is a constant
 $C_2>0$ and $\rho_* < \bar\rho$ such that
 \[
|\nabla U(x)|^2 \le C_2 U(x),\qquad x\in \bar\Omega_{\rho_*}.
 \]
Since $U$ is of the class $C^1$ and $U(x)>0$ on the compact set
$\bar\Omega_{\bar\rho}\setminus\Omega_{\rho_*}$, we can make  $C_2$
larger if necessary so that
\[
|\nabla U(x)|^2 \le C_2 U(x),\qquad x\in
\bar\Omega_{\bar\rho}\setminus\Omega_{\rho_*},
\]
i.e., there exists a constant $C_2=C_2(\bar\rho)>0$ such that
\[
|\nabla U(x)|^2 \le C_2 U(x),\qquad x\in \bar\Omega_{\bar\rho}.
\]
It follows that
 \[
a^{ij}_\alpha(x)\partial_i U(x)\partial_j U(x)\le C_2
\Lambda_\alpha\rho =: H^\alpha(\rho),\qquad\; x\in
\Omega_\rho,\;\rho\in [0,\bar\rho],
 \]
where $ \Lambda_\alpha=\Lambda_\alpha(\bar\rho):=\sup_{x\in
{\Omega_{\bar\rho}}}  |A_\alpha(x)|$.  We note that, since each
$\mu_\alpha$ is a regular stationary measure and
$\bar\Omega_{\rho_m}=\{0\}$, $\mu_\alpha(\bar\Omega_{\rho_m})=0$.
Applying Proposition \ref{aunif} with respect to the stationary
measure
 \[
 \tilde\mu_\alpha=\displaystyle\frac{\mu_\alpha|_{{\Omega_{\bar\rho}}}} {\mu_\alpha({\Omega_{\bar\rho}})}
 \]
 corresponding to ${\cal L}_{A_\alpha}$ in ${\cal W}$,  we have
 \begin{equation}\label{P4}
\mu_\alpha(\Omega_{\rho_0})=\mu_\alpha(\Omega_{\rho_0}\setminus
\bar\Omega_{\rho_m})\le \rme^{-C\int_{\rho_0}^{\bar\rho}\frac
{1}{\rho}\rmd \rho}, \qquad \rho_0\in (0,\bar\rho),
\end{equation}
where
$C=C(\bar\rho)=\frac{\lambda_U^{\bar\rho}}{C_2(\bar\rho)}\inf_{\alpha}
\frac {\lambda_\alpha^{\bar\rho}}{\Lambda_\alpha(\bar\rho)}$, which
is a  positive constant since $\cal A=\{A_\alpha\}$ is a normal null
family.

It follows from Theorem~\ref{t410} that $\cal M_{\cal
A}\neq\emptyset$, and that any ${\cal A}$-limit measure is an
invariant measure of $\varphi^t$  supported on the global attractor
$\cal J$. Let $\mu\in {\cal M}_{\cal A}$. Since $\Omega_{\rho_0}$ is
open, we have by applying Proposition \ref{wt-kh} to \eqref{P4} and
taking limit $A_\alpha\to 0$ that
 \[
\mu(\Omega_{\rho_0})\le  \rme^{-C\int_{\rho_0}^{\bar\rho}\frac
{1}{\rho}\rmd\rho }, \qquad \rho_0\in (0,\bar\rho).
\]
By letting $\rho_0\to 0$ in the above, we have
\begin{equation}\label{mu00}
\mu(\{0\})=0.
\end{equation}

 For any given $\rho\in (0,\bar\rho)$, we have by \eqref{P3} and
Proposition~\ref{UNIF-S} that $\{{\cal L}_{A_\alpha}\}$ admits a
uniform anti-Lyapunov function in   $\Omega_\rho$. It also follows
from \eqref{P3} that $\{0\}$ is the maximal repeller of $\varphi^t$
in $\Omega_\rho$. Applying Theorem~\ref{anti-L} b) and \eqref{mu00},
we have
\[
\mu(\Omega_\rho)=\mu(\Omega_\rho\setminus \{0\})=0.
\]
Since $\rho$ is arbitrary, $\mu({\Omega_{\bar\rho}})=0$. Since
$\bar\rho\in(0,\rho_M)$ is also arbitrary, $\mu(\cal W)=0$.

Since ${\rm supp}(\mu)\subset\cal J$ and  $\cal J\setminus {\cal W}$
is a positively invariant set of $\varphi^t$, it follows from
Proposition \ref{invariant-measure} that $\mu$ is actually supported
on the $\omega$-limit set ${\cal J}_*=\omega(\cal J\setminus {\cal
W})$. The $\cal A$-stability of ${\cal J}_*$ follows from
Proposition~\ref{stable}, Remark~\ref{rk2.3} 1), and
Lemma~\ref{tight} a).
\end{proof}

%{\color{red}
%\begin{Remark}\label{comp}\rm
%Note that in the proof of Theorem \ref{equili}, the assumption that $\varphi^t$ is negatively complete or equivalently a flow on $\cal U$
%is only used to conclude ${\rm supp}\mu\subset \cal J$ by Corollary \ref{LS}. So if ${\rm supp}\mu\subset \cal J$ is known, the result of
%Theorem \ref{equili} and hence Corollaries \ref{cor42}, \ref{cor43} remains valid without this assumption.
%\end{Remark}
%}

\begin{Corollary}\label{cor42}
Assume  $V\in C^1(\cal U,\R^n)$ and that \eqref{ode} admits a $C^2$
Lyapunov function with bounded second derivatives on $\cal U$. Let
$x_0\in\cal U$ be an equilibrium of $\varphi^t$ such that all
eigenvalues of $DV(x_0)$ have positive real parts. Then $x_0$ is a
strongly repelling equilibrium, and consequently the conclusion of
{\rm Theorem~\ref{equili}} holds for $x_0$.
\end{Corollary}
\begin{proof} By standard theory of dissipative dynamical systems \cite[Sections 17 and 18]{LaL}, $x_0$
is a strongly repelling equilibrium. In fact,  there is a positive
definite $n\times n$ matrix  $B$ such that
\[
{(DV(x_0))}^\top B+B(DV(x_0))=I.
\]
Let $U(x)=(x-x_0)^\top B(x-x_0)$. Then $U$ is a desired entire weak
anti-Lyapunov function satisfying \eqref{P1}-\eqref{P3} in the
vicinity of $x_0$.
\end{proof}

\begin{Corollary}\label{cor43}
Assume  $V\in C^1(\cal U,\R^n)$ and that \eqref{ode} admits a $C^2$
Lyapunov function with bounded second derivatives on $\cal U$. Then
{\rm Theorem~\ref{equili}} and {\rm Corollary~\ref{cor42}} hold for
any family of the form $\{\epsilon A\}$ when $\epsilon\to 0$, where
$A\in \tilde{\cal A}$ is such that $\sup_{x\in \cal
U} |A(x)|<\infty$ and $A(x)\to 0$ as $x\to
\partial{\cal U}$ when $\cal U\ne \R^n$.
\end{Corollary}

\begin{proof} {\rm Theorem~\ref{equili}} and {\rm Corollary~\ref{cor42}} clearly hold because
 $\{\epsilon A\}$ is an invariant normal null family.
\end{proof}

\section{An example of stochastically stable limit cycles}

We demonstrate by an example the application of our results  in
obtaining stochastically stable limit cycles, in connection with
stochastic Hopf bifurcations.

Consider stochastic differential equations
\begin{equation}\label{Hopf}
\left\{\begin{array}{l}\rmd x=(bx-y-x(x^2+y^2))\rmd t+ \sqrt{\epsilon}
 g^{11}(x,y)\rmd W_1+ \sqrt{\epsilon} g^{12}(x,y)\rmd W_2,\\
\rmd y=(x+by-y(x^2+y^2))\rmd t+\sqrt{\epsilon}  g^{21}(x,y)\rmd W_1+
\sqrt{\epsilon} g^{22}(x,y)\rmd W_2,\end{array}\right.
\end{equation}
where $b,\epsilon$ are parameters, $g^{ij}\in W^{1,2\bar
p}_{loc}(\R^2)$, $i,j=1,2$, for some ${p}>2$, and
$G(x,y)=(g^{ij}(x,y))$ is everywhere non-singular and  bounded  on
$\R^2$. Denote $V_b(x,y)=(bx-y-x(x^2+y^2),x+by-y(x^2+y^2))^\top$
 and $A(x,y)=\frac {G(x,y) G^\top(x,y)}2$.
Also denote the adjoint Fokker-Planck operators associated with
\eqref{Hopf} by ${\cal L}_{V_b,\epsilon A}$.

Let $U(x,y)=x^2+y^2$. Then it is clear that
\begin{equation}\label{Vb}
V_b\cdot \nabla U=2U(b-U)
\end{equation}
which approaches to $ -\infty$ as $x^2+y^2\to\infty$. It follows
that, for each $b$, the unperturbed system of \eqref{Hopf} is
dissipative; in particular, it generates a (positive) semiflow
$\varphi^t_b$ on $\R^2$ which  admits a global attractor ${\cal
J}_b$. Moreover, since $A(x,y)$ is bounded, Proposition~\ref{UNIF-S}
implies that $U$ is in fact a uniform Lyapunov function in $\R^2$
with respect to the family $\{{\cal L}_{V_b,\epsilon A}\}$ when
$0<\epsilon\ll 1$.  We have by Proposition \ref{exist} that each
${\cal L}_{V_b,\epsilon A}$ admits a unique stationary measure
$\mu_{b,\epsilon A}$,  by Lemma \ref{tight} a) that the set of limit
measures of $\{\mu_{b,\epsilon A}\}$ as $\epsilon\to 0$ is
non-empty, and by Theorem~\ref{t410} that each limit measure is an
invariant  measure of $\varphi^t_b$ supported on  ${\cal J}_b$.

When $b\le 0$, ${\cal J}_b=\{0\}$. Hence $\mu_{b,\epsilon A}\to
\delta_0$ as $\epsilon\to 0$. Consequently, $\{0\}$ and $\delta_0$
are $\{\epsilon A\}$-stable.

When $b>0$, ${\cal J}_b=\bar \Omega_b$ - the closed disk of radius
$\sqrt b$ centered at the origin. We note that $U$ is also an entire
weak anti-Lyapunov function in $\Omega_b$ satisfying
\eqref{P1}-\eqref{P3} for $x_0=0$. An application of Corollary
\ref{cor43} in the open disk $\Omega_b$ yields that each limit
measure of $\{\mu_{b,\epsilon A}\}$ is supported on
$\partial\Omega_b=: C_b$ - the circle with radius $\sqrt b$, as
$\epsilon\to 0$. Since $C_b$ is a periodic invariant cycle of
$\varphi^t_b$, it is uniquely ergodic with the Haar measure $\mu_b$
as the only invariant measure. Therefore, $\mu_{b,\epsilon A} \to
\mu_b$ as $\epsilon\to 0$. Consequently, $C_b$ and $\mu_b$ become
$\{\epsilon A\}$-stable in this case.

This example thus demonstrates  an interesting  phenomenon of
stochastic Hopf bifurcation in which the repelling equilibrium
$\{0\}$  is invisible when $b>0$ under noise perturbations.

\section{Appendix}

In this  section, we will summarize some general notions and
fundamental properties of dynamics of  ordinary differential
equations including dissipation and invariant measures.  Some
materials concerning dissipation are taken from the Appendix of
\cite{JM2} for a general domain $\cal U\subset \R^n$.

Consider \eqref{ode} for a continuous $V$. Throughout the section,
we assume that solutions of \eqref{ode} are locally unique (e.g.,
$V$ is locally Lipschitz continuous). Then \eqref{ode} generates a
(continuous) {\em local flow} $\varphi^t: \cal U\to \cal U$, where
for each $\xi\in \cal U$, $\varphi^t(\xi)$ is the solution of
\eqref{ode} for $t$ lying in the maximal interval ${\cal I}_\xi$ of
existence in $\cal U$.  A set  $B\subset \cal U$ is called an {\em
invariant set} (resp. {\em positively invariant set}, resp. {\em
negatively invariant set}) of $\varphi^t$ if ${\cal I}_\xi=\R$
(resp. ${\cal I}_\xi \supset\R_+$, resp. ${\cal I}_\xi
\supset\R_{-}$) for all $\xi\in B$, and $\varphi^t(B)=B$ for all
$t\in \R$ (resp. $\varphi^t(B)\subset B$ for all $t\in \R_+$, resp.
$\varphi^t(B)\subset B$ for all $t\in \R_{-}$). If $\cal U$ itself
is an {\em invariant set} (resp. {\em positively invariant set},
resp. {\em negatively invariant set}) of $\varphi^t$, $\varphi^t$ is
called a {\em flow} (resp. {\em positive semiflow}, resp. {\em
negative semiflow}). We note that if $\varphi^t$ is a flow, then for
each $t\in \R$, $\varphi^t: \cal U\to \cal U$ is a homeomorphism.

\subsection{Dissipation v.s. anti-dissipation}
 For any
set $B\subset \mathcal{U}$ such that $\cup_{t\in\R_+}\varphi^t(B)$
(resp. $\cup_{t\in\R_-}\varphi^t(B)$) exists and is pre-compact in
$\mathcal{U}$, we recall that the {\em $\omega$-limit set of $B$}
(resp. {\em $\alpha$-limit set of $B$}) is defined as
\[
\omega(B)=\cap_{\tau\ge 0}\overline{\cup_{t\ge \tau}\varphi^t(B)}
\qquad\qquad {\rm (resp.} \; \alpha(B)=\cap_{\tau\le
0}\overline{\cup_{t\le \tau}\varphi^t(B)}{\rm )}.
\\
\\
\]

\begin{Definition} {\em Let  $\Omega\subset\cal U$ be a  connected  open set.

1) $\varphi^t$ is said to be {\em dissipative} (resp. {\em
anti-dissipative}) {\em in $\Omega$ } if $\Omega$ is a positively
(resp. negatively) invariant set of $\varphi^t$ and there exists a
compact subset $B$ of $\Omega$ with the property that for any
$\xi\in \Omega$ there exists a $t_0(\xi)>0$ such that
$\varphi^t(\xi)\in B$ as $t\ge t_0(\xi)$ (resp. $t\le -t_0(\xi)$).

2) When $\Omega$ is a positively (resp. negatively) invariant set of
$\varphi^t$, the {\em maximal attractor $\cal J$} (resp. {\em
maximal repeller $\cal R$}) {\em of $\varphi^t$ in $\Omega$} is a
compact invariant set of $\varphi^t$ which attracts (resp. repels)
any pre-compact subset $K$ of $\Omega$, i.e., $\omega(K)\subset\cal
J$ (resp. $\alpha(K)\subset\cal R$), or equivalently,
$\lim_{t\to+\infty} {\rm dist} (\varphi^t(K),\cal J)= 0$ (resp.
$\lim_{t\to -\infty} {\rm dist} (\varphi^t(K),\cal R)= 0$),  where
for any two subsets $A,B$ of $\R^n$, ${\rm dist} (A, B)$ denotes the
Hausdorff semi-distance from  $A$ to $B$.

3) Suppose that  the  maximal attractor $\cal J$   (resp. maximal
repeller $\cal R$) of $\varphi^t$ exists in $\Omega$.  If
$\Omega=\cal U$, then we call $\cal J$ (resp. $\cal R$) the {\em
global attractor} (resp. {\em global repeller}) {\em of
$\varphi^t$}. Otherwise, it is called a {\em local attractor} (resp.
{\em local repeller}) {\em of $\varphi^t$}.}
\end{Definition}

It is clear that  maximal attractor or repeller of $\varphi^t$  in
$\Omega$, if exists, must be unique. Consequently, global attractor
or repeller of $\varphi^t$  is unique.

%Let ${\rm dist} (A, B)= \sup_{a\in A}\inf_{b\in B} |a-b|$ denote the
%Hausdorff distance from a set $A$ to a set $B$. Then the following
%property is well-known.

\medskip

\begin{Proposition}\label{omega-set1} Let $\Omega\subset \cal U$ be a connected open set.

{\rm 1)} $\varphi^t$ is dissipative {\rm (}resp.
anti-dissipative{\rm )} in $\Omega$   if and only if $\Omega$ is a
positively {\rm (}resp. negatively{\rm )} invariant set and admits a
maximal attractor {\rm (}resp. repeller{\rm )} of $\varphi^t$.

{\rm 2)} If $\varphi^t$ is dissipative {\rm (}resp.
anti-dissipative{\rm )} in $\Omega$, then the maximal attractor
$\cal J$ {\rm (}resp. maximal repeller $\cal R${\rm )} of
$\varphi^t$ in $\Omega$ equals
\[
\cal J=\bigcup_{B\subset \Omega \ {\rm pre-compact}}\omega(B)\;\;\;
\hbox{
  {\rm (}resp. }  \cal R=\bigcup_{B\subset \Omega \ {\rm
pre-compact}}\alpha(B){\rm )}. \] %and, if $\Omega$ itself is
%pre-compact in $\mathcal{U}$, then
%\[
%\cal J=\omega(\Omega)\;\;\;\hbox{ {\rm (}resp. } \cal
%R=\alpha(\Omega){\rm )}.
%\]
\end{Proposition}

\begin{proof} See Proposition~6.2 in \cite{JM2}.
\end{proof}

\medskip

\begin{Definition}\label{D2} {\em Let $\Omega\subset \cal U$ be a  connected open
set and  $U\in C^1(\Omega)$ be a compact function with essential
upper bound $\rho_M$. For each $\rho\in  [0,\rho_M)$,
 denote $\Omega_\rho$
 as the $\rho$-sublevel set of $U$.
\begin{itemize}

\item[{\rm 1)}] $U$ is a
 {\em weak Lyapunov function} (resp. {\em weak anti-Lyapunov
function})  {\em of \eqref{ode}  in $\Omega$} if
\begin{equation}\label{U33}
V(x)\cdot \nabla U(x)\leq 0 \quad\quad \hbox{(resp.}\, \ge 0{\rm
)},\;\;\, x\in \tilde\Omega=\Omega\setminus \bar\Omega_{\rho_m},
\end{equation}
  where $\rho_m\in (0,\rho_M)$ is a constant, called
{\em essential lower bound of $U$}, and $\tilde{\Omega}$ is called
{\em essential domain of  $U$ in $\Omega$}.

If \eqref{U33} holds for all $x\in \Omega$ instead of
$\tilde\Omega$, then
 $U$ is called an {\em
entire weak Lyapunov function {\rm(resp.}  entire weak anti-Lyapunov
function{\rm)} of \eqref{ode}  in $\Omega$}.

\item[{\rm 2)}]  $U$ is  a {\em Lyapunov function} (resp. {\em
anti-Lyapunov function}) {\em of \eqref{ode}  in $\Omega$} if there
exists a constant $\gamma>0$, called a {\em Lyapunov constant of
$U$},  such that
\begin{equation}\label{U3c}
V(x)\cdot \nabla U(x)\leq -\gamma\quad\quad \, \hbox{(resp.}\, \ge
\gamma{\rm )},\;\;\,  x\in \tilde\Omega=\Omega\setminus
\bar\Omega_{\rho_m},
\end{equation}
 where  $\rho_m\in (0,\rho_M)$  is a constant,  called  {\em
essential lower bound of $U$}, and $\tilde{\Omega}$ is called {\em
essential domain of $U$ in $\Omega$}.

\end{itemize}
}
\end{Definition}
\medskip

%If \eqref{ode} does generate  a local flow
%$\varphi^t$, then we say that $U$ is a weak Lyapunov (resp. weak
%anti-Lyapunov), entire weak Lyapunov, Lyapunov (resp. anti-Lyapunov)
%function of $\varphi^t$ in $\Omega$ whenever applicable in the
%above.
\medskip

\begin{Proposition}\label{dissip}  Let $\Omega\subset \cal U$
be a connected open set.

{\rm a)} If \eqref{ode} admits a weak Lyapunov  {\rm (}resp.
anti-Lyapunov{\rm )} function in $\Omega$, then $\Omega$ must be a
 positively {\rm (}resp. negatively{\rm )}
invariant set of $\varphi^t$.

{\rm b)}  If \eqref{ode} admits  a Lyapunov  {\rm (}resp.
anti-Lyapunov{\rm )} function in $\Omega$, then it must be
 dissipative {\rm (}resp. anti-dissipative{\rm )} in $\Omega$, with
the  maximal
 attractor {\rm (}resp. repeller{\rm )} in $\Omega$ being
 $\omega(\Omega_{\rho_m})$ {\rm (}resp. $\alpha(\Omega_{\rho_m})${\rm
 )}, where $\rho_m$ is the essential lower bound and $\Omega_{\rho_m}$ is the $\rho_m$-sublevel set
 of the  Lyapunov
 {\rm (}resp. anti-Lyapunov{\rm )} function.
\end{Proposition}
\begin{proof}
See Proposition~6.3 in \cite{JM2}.
\end{proof}

\medskip

The following is a strong version of the well-known LaSalle
invariance principle (\cite[Chapter 2, Theorem 6.4]{La}) for
locating $\omega$-limit sets
%(resp. $\alpha$-limit sets)
in a domain which admits an entire weak Lyapunov
%(resp. anti-Lyapunov)
function.
\medskip

\begin{Proposition}\label{LaSalle} Let $\Omega\subset\cal U$ be a
connected open set. If \eqref{ode} admits  an entire
 weak Lyapunov
 %{\rm (}resp. anti-Lyapunov{\rm )}
 function $U$ in $\Omega$, then for
any $x_0\in \Omega$,
\[
\omega(x_0)\subset S,
%\;\; \qquad {\rm (}resp.\; \alpha(x_0)\subset S{\rm )},
\]
 where $S=\{x\in  \Omega: V(x)\cdot \nabla U(x)=0\}$.
\end{Proposition}

We note that the original LaSalle invariance principle does not
require $U$ be  a compact function (in this case, the conclusion of
Proposition~\ref{LaSalle} only holds for  points which have bounded
 positive orbits).
\medskip

\begin{Definition}\label{D22} {\em   A compact invariant
set $\cal J\subset \cal U$ (resp.  $\cal R\subset \cal U$) of
$\varphi^t$ is called a {\em strong  local attractor} (resp. {\em
strong local repeller}) of $\varphi^t$ if there is a  connected,
positively {\rm (}resp. negatively{\rm )} invariant neighborhood
$\cal W$ of $\cal J$ (resp. $\cal R$) in $\cal U$, called an {\em
isolating neighborhood}, with oriented $C^2$ boundary such that
\begin{itemize}
\item[{\rm 1)}] $\cal J$ attracts $\cal W$ (resp. $\cal R$ repels $\cal W$);
% i.e., $\omega(\cal W)=\cal J$ (resp. $\alpha(\cal W)=\cal R$);
\item[{\rm 2)}] $V(x)\cdot \nu(x)<0$
(resp. $V(x)\cdot \nu(x)>0$), $x\in\partial{\cal W}$,  where
$\nu(x)$ denotes the outward unit normal vector of $\partial{\cal
W}$ at $x\in
\partial{\cal W}$.
\end{itemize}}
\end{Definition}

\begin{Remark}\label{R32} {\em
%1) In Definition~\ref{D22}, $\partial\cal W$ may be assumed to be
%only $C^1$ because any $C^1$ boundary with finitely many components
%can be approximated by a $C^2$ boundary.
The neighborhood $\cal W$ in Definition~\ref{D22} is indeed an
isolating neighborhood in the usual sense, i.e.,  any entire orbits
of $\varphi^t$ lying in $\cal W$ must lie in $\cal J$ (resp. $\cal
R$). }
\end{Remark}
\medskip

\begin{Proposition}~\label{attractor-repeller} A local attractor {\rm (}resp.  repeller{\rm )} of $\varphi^t$ is a
 strong local attractor
  {\rm (}resp.  repeller $\cal R${\rm )}  if and only if it has an isolating  neighborhood  on which   \eqref{ode} admits
  a
Lyapunov  {\rm (}resp. anti-Lyapunov{\rm )} function.
\end{Proposition}
\begin{proof} We only treat the case of a local attractor $\cal J$ of $\varphi^t$.

To show the necessity, we parametrize the boundary $\partial{\cal
W}$ of an isolating neighborhood $\cal W$ in Definition~\ref{D22} by
a $C^2$ function $U$, i.e., $\partial{\cal W}=\{x\in \cal U:
U(x)={\rm constant}\}$. Let $\gamma=\min_{x\in
\partial{\cal W}} |V(x)\cdot \nu(x)|$. Then $\gamma>0$.  Since
$\nu(x)=\nabla U(x)/|\nabla U(x)|$,
\[V(x)\cdot \nabla U(x)\le -\gamma|\nabla U(x)|\le -\gamma \min_{x\in
\partial {\cal W}}|\nabla U(x)|=: -\tilde \gamma
\]
for all $x\in
\partial {\cal W}$. It follows that there is a  neighborhood
$\tilde{\cal W}$ of $\partial{\cal W}$ in $\cal W$ such that
\[V(x)\cdot \nabla U(x)<-\frac{\tilde\gamma}2, \qquad \quad x\in \tilde{\cal W},
\]
i.e.,   $U$ becomes a Lyapunov function of \eqref{ode} in $\cal W$.

To prove the sufficiency, we let $\cal W$ be an isolating
neighborhood of  $\cal J$ on which a Lyapunov function $U$ of
\eqref{ode} exists. Using $C^2$ approximation, we may assume without
loss of generality that $U$ is of the class $ C^2$ in a connected
open set $\tilde{\cal  W}\supset {\cal W}$ on which $U$ remains as a
Lypunov function of $\varphi^t$. Then it is clear that
Definition~\ref{D22} 2) is satisfied on $\partial{\cal W}$.
\end{proof}
\medskip

\begin{Definition}~\label{D220} {\em An equilibrium  $x_0$ of \eqref{ode} is called a (local) {\em strongly
repelling equilibrium}  if there is a connected open set $\cal W$
containing $x_0$, called an {\em isolating neighborhood of $x_0$},
and a compact function $U\in C^2(\cal W)$ such that
\begin{eqnarray}
&& \hbox{the Hessian matrix } D^2U\; \hbox{ is everywhere positive definite in} \, \cal W,\label{P1}\\
&& U(x_0) =0,\quad \nabla U(x_0)=0; \;\quad U(x) > 0, \,\;  x\in
\cal W\setminus
\{x_0\},\,\; \hbox{and}\label{P2}\\
&&V(x)\cdot\nabla U(x)>0,\,\;  x\in \cal W\setminus
\{x_0\}.\label{P3}
\end{eqnarray}
} \end{Definition}
\medskip

\begin{Remark} {\em 1) We note that a strongly repelling equilibrium  $x_0$ of \eqref{ode} is necessarily a strong
local repeller.

2) By a time reversing application of Converse Lyapunov Theorem (see
e.g., \cite[Theorem 4.2.1]{LWY}, \cite[Theorem V.2.12]{BS}), if
$x_0$ is a (local) repelling equilibrium in the usual sense, then
there is a Lipschitz continuous function $U$ in a neighborhood of
$x_0$ which is an almost everywhere entire weak anti-Lyapunov
function and satisfies \eqref{P2}, \eqref{P3} almost everywhere.}
\end{Remark}

\medskip

\subsection{Invariant measures}  For a Borel set $\Omega\subset\R^n$, we denote
by $M(\Omega)$ the set of Borel probability measures on $\Omega$
furnished with the weak$^*$-topology.

For an invariant (resp. positively invariant, negatively invariant)
Borel set $\Omega\subset\cal U$ of $\varphi^t$, $\mu\in M(\Omega)$
is said to be an {\em invariant measure} (resp. {\em positively
invariant measure, negatively invariant measure}) of $\varphi^t$ on
$\Omega$  if for any Borel set $B \subset \Omega$,
\begin{equation}\label{invariant}
\mu \left( \varphi^{-t} (B) \right) = \mu (B),\qquad\qquad  t \in
\R\;\; {\rm (resp.}\; t\in\R_+,\;t\in \R_{-}{\rm )},
\end{equation}
where $\varphi^{-t}(B)=\{x\in \Omega: \varphi^t(x)\in B\}$.

 We note that if $\Omega$ is an invariant Borel set, then it
follows from the flow property that any positively or negatively
invariant measure on $\Omega$ must be invariant.

%It is well-known that any compact invariant subset supports at least
%one invariant  measure. An invariant measure defined in a Borel
%subset can be naturally extended to an invariant  measure of the
%entire phase space. Conversely, the support of any invariant measure
%on the entire phase space must be a closed invariant set. Therefore,
%one does not need to distinguish an invariant   measure defined on
%the entire phase space with those defined on Borel subsets.

An invariant  measure $\mu$ is an {\em ergodic measure} if any
invariant Borel set has either full $\mu$-measure or zero
$\mu$-measure. A compact invariant set admits at least one ergodic
measure. It is said to be {\em uniquely ergodic} if it admits only
one invariant
 measure.
\medskip

\begin{Proposition}\label{invariant-measure} Let $\Omega\subset\cal U$ be a
 positively {\rm (}resp. negatively{\rm )} invariant Borel
set in $\mathcal{U}$ which admits a positively {\rm (}resp.
negatively{\rm )} invariant  measure $\mu$ of $\varphi^t$. Then the
following holds.
\begin{itemize}
\item[{\rm 1)}] If $\Omega$ is pre-compact in $\cal U$, then $\mu$ is
an invariant measure of $\varphi^t$ supported on
$\omega(\Omega)$ {\rm (}resp. $\alpha(\Omega)${\rm )}.

\item[{\rm 2)}]  If $\Omega$ is a connected open subset of $\cal U$ and $\varphi^t$ is
dissipative {\rm (}resp. anti-dissipative{\rm )}  in
$\Omega$, then $\mu$ is supported on the maximal attractor {\rm
(}resp. repeller{\rm )} of $\varphi^t$ in $\Omega$.
\end{itemize}
\end{Proposition}

\begin{proof} 1) can be easily seen from the definitions
of invariant  measures, positively (resp. negatively) invariant
measures,  $\omega$-limit (resp. $\alpha$-limit)  sets.

To prove 2), we first consider the case that $\varphi^t$ is
dissipative in $\Omega$. Let $\cal J$ denote the maximal attractor
of $\varphi^t$ in $\Omega$. Since $\cal J$ is a compact invariant
set, it is sufficient to show that $\mu(\cal J)=1$.

For any given $\epsilon>0$, we have by the Borel regularity of $\mu$
that there exist a compact subset $B$ of $\Omega$ and an open
neighborhood $W$ of $\mathcal{J}$ lying in $\Omega$ such that
$\mu(B)>1-\epsilon$ and $\mu(W\setminus \mathcal{J})<\epsilon$.
Since $\mathcal{J}$ attracts $B$, there exists $t>0$ such that
$\varphi^t(B)\subseteq W$. Thus,
\[
\mu(\mathcal{J})\ge \mu(W)-\epsilon\ge
\mu(\varphi^t(B))-\epsilon=\mu \big(
(\varphi^t)^{-1}(\varphi^t(B)\big)-\epsilon\ge \mu(B)-\epsilon\ge
1-2\epsilon.
\]
Since $\epsilon$ is arbitrary, $\mu(\mathcal{J})=1$.

The case when  $\varphi^t$ is anti-dissipative in $\Omega$ is
similar.
\end{proof}

\end{document}